%% file: ass_final.tex
\documentclass{lms}

\usepackage{amsfonts}
\usepackage{amssymb}
\usepackage{amscd}
\usepackage{amsmath}
\usepackage{graphicx}
\usepackage{epsfig}  
\usepackage{pstricks}
\usepackage{color}	
\usepackage{rotating}

\input{xy}
\xyoption{poly}
\xyoption{2cell}
\xyoption{all}

\newcommand{\ot}{\leftarrow}

\newcommand{\longto}{\longrightarrow}

\newcommand{\ralf}{}

\newcommand{\za}{\alpha}
\newcommand{\zb}{\beta}

\newcommand{\zD}{\Delta}

\newcommand{\zg}{\gamma}
\newcommand{\zG}{\Gamma}

\newcommand{\zs}{\sigma}

\newcommand{\Aut}{\textup{Aut}\,\mathcal{A}}
\newcommand{\aut}{\textup{Aut}^+\mathcal{A}}
\newcommand{\homeo}{\textup{Homeo}^+(S,M)}
\newcommand{\homeozero}{\textup{Homeo}_0(S,M)}

\newcommand{\homeos}{\textup{Homeo}^+(S,\partial S)}
\newcommand{\mg}{\mathcal{MG}(S,M)}
\newcommand{\mgm}{\mathcal{MG}_{\bowtie}(S,M)}
\newcommand{\Mod}{\mathcal{M}od\,S}

\newcommand{\ZZ}{\mathbb{Z}}
\newcommand{\fbar}{\overline{f}}

\newcommand{\gbar}{\overline{g}}
\newcommand{\mubar}{\overline{\mu}}

\def\amas{\mathbf{x}}
\def\Autp{{\rm Aut}^+ {\cal A}}
\def\A{\mathbb A}

\newtheorem{theorem}{Theorem}[section] 
\newtheorem{proposition}[theorem]{Proposition}

\newtheorem{corollary}[theorem]{Corollary}
\newtheorem{lemma}[theorem]{Lemma}

\newnumbered{definition}{Definition}
\newnumbered{example}[theorem]{Example}
\newnumbered{xca}[theorem]{Exercise}

 \newnumbered{conj}{Conjecture}
\newnumbered{remark}[theorem]{Remark}

\title{Cluster Automorphisms}
\author{Ibrahim Assem, Ralf Schiffler and Vasilisa Shramchenko}
\classno{16S99 (primary), 16S70, 16W20 (secondary)}
\extraline{The first and third authors are supported by the NSERC of Canada, the FQRNT of Qu\'ebec and the Universit\'e de Sherbrooke, the second author is supported by the NSF grants  DMS-0908765 and DMS-1001637 and by the University of Connecticut.}
\begin{document}
\maketitle
\begin{abstract}
In this article, we introduce the notion of cluster automorphism of a given cluster algebra as a $\ZZ$-automorphism of the cluster algebra that sends a cluster to another and commutes with mutations. We study the group of cluster automorphisms in detail for acyclic cluster algebras and cluster algebras from surfaces, and we compute this group explicitly for the Dynkin types and the euclidean types. 
\end{abstract}
\begin{section}{Introduction} Cluster algebras were introduced by Fomin and Zelevinsky in \cite{FZ1,FZ2} in the context of canonical bases and total positivity. These are $
\mathbb{Z}$-algebras whose generators are grouped into sets called clusters,
and one passes from one cluster to another using an operation called
mutation.

We are interested in the question whether one can study cluster
algebras as a category. This means understanding the morphisms which
preserve their very particular structure, that is, which keep invariant the
grouping of generators into clusters and are compatible with mutations. 
As a
first step in this direction, we study here what we call cluster
automorphisms. We define a cluster automorphism of
a given cluster algebra as an
automorphism of $\mathbb{Z}$-algebras sending a cluster to another and
commuting with mutations. 
 Thus, in this paper, we study the symmetries of a given cluster
algebra and compute the cluster automorphism group for the best known
classes of cluster algebras, those arising from an acyclic quiver and those
arising from a surface.
Observe that, in \cite{FZ2}, Fomin and Zelevinsky have considered a related notion of strong isomorphisms, by which they mean an isomorphism of the cluster
algebras which maps every seed to an isomorphic seed. As will follow from
our results, a strong automorphism of a cluster algebra is what we call here a
direct cluster automorphism.

Let $\mathcal{A}=\mathcal{A}(\mathbf{x},Q)$ be a cluster algebra. Among the most interesting properties of the automorphism group $\Aut $ of $\mathcal{A}$ is the fact that an element of this group sends the quiver $Q$ either to itself or to  the opposite quiver $Q^{\textup{op}}$. This allows to define a subgroup $\aut$ of $\Aut$ consisting of those automorphisms sending $Q$ to itself. 
We prove that the index of $\aut$ in $\Aut$ is two if and only if $Q$ is mutation equivalent to $Q^{\textup{op}}$ and otherwise $\Aut = \aut$.
We first compute these groups in the context of acyclic cluster algebras. In this case, the combinatorics of the cluster algebra is nicely encoded in the cluster category introduced in \cite{BMRRT} and, for type $\mathbb{A}$, also in \cite{CCS}. In particular, we recall that  the Auslander-Reiten quiver of the cluster category of an acyclic cluster algebra has a particular connected component, called the transjective component.

\begin{theorem}
Let $\mathcal{A}$ be an acyclic cluster algebra and $\zG_{tr}$ the transjective component of the Auslander-Reiten quiver of the associated cluster category. Then
$\aut$ is the quotient of the group $\textup{Aut}(\zG_{tr})$ of the quiver automorphisms of $\zG_{tr}$, modulo the stabiliser $\textup{Stab}(\zG_{tr})_0$ of the points of this component. Moreover, if $\zG_{tr} \cong \mathbb{Z}\Delta$, where $\Delta$ is a tree or of type $\tilde{\mathbb A}$ then $\Aut = \aut \rtimes \ZZ_2$ and this semidirect product is not direct.
\end{theorem}

As an easy consequence, we compute the automorphism groups of the cluster algebras of Dynkin and euclidean types.

We next consider the case of the cluster algebras arising from an oriented marked surface $(S,M)$ with $p$ punctures in the sense of Fomin, Shapiro and Thurston \cite{FST} (see also \cite{FG}). We define the marked mapping class group $\mgm$ of $(S,M)$ to be a semidirect product of a power of $\ZZ_2$ and the quotient $\mg $  of the group $\homeo$ of orientation preserving homeomorphisms from the surface $S$ to itself which map the set of marked points $M$ to itself, modulo the subgroup $\homeozero$ consisting of those $f\in \homeo$ which are isotopic to the identity via an isotopy that fixes $M$ pointwise. We then prove our second theorem.

\begin{theorem}
Let $(S,M)$ be a marked surface. Then $\mgm$ is isomorphic to a subgroup of $\aut$. Furthermore, if $(S,M)$ is a disc or an annulus without punctures, then $\mg\cong\aut$.
\end{theorem}

We also consider the case of the disc with one or two punctures and show that (except for two exceptional cases) we also have $\mgm\cong\aut$.

As a consequence of our results, we show that if the cluster algebra $\mathcal{A}$ is acyclic or arising from a surface, then the group $\Aut$ is finite if and only if $\mathcal{A}$ is a cluster algebra of Dynkin type.

The paper is organised as follows. In section \ref{sect 1}, we define our notion of cluster automorphism and prove some of its elementary properties, section \ref{sect_acyclic} is devoted to the case of acyclic cluster algebras and section \ref{sect 3} to that of cluster algebras arising from  surfaces. Finally in section \ref{sect 4}, we consider the finiteness of the automorphism group.

{\bf Acknowledgements.}{The authors thank the referee for pointing out an error in an earlier version of the paper, as well as for many useful suggestions.}
\end{section}

\section{Cluster automorphisms}\label{sect 1}

\subsection{ Cluster algebras}
We recall that a {\it quiver} is a quadruple $Q=(Q_0,Q_1,s,t)$ consisting of two sets: $Q_0$ (whose elements are called \emph{points}) and $Q_1$ (whose elements are called \emph{arrows}), and of two maps $s,t:Q_1\to Q_0$ associating to each $\alpha \in Q_1$ its \emph{source} $s(\alpha)$ and its \emph{target} $t(\alpha)$, respectively. Given a point $i\in Q_0$, we denote by $i^-=\{ \alpha\in Q_1\mid t(\alpha)=i \}$ the set of arrows ending in $i$, and by $i^+=\{ \alpha\in Q_1\mid s(\alpha)=i  \}$ the set of arrows starting in $i$. 

Let now $Q$ be a connected finite quiver without oriented cycles of length one or two. Let $n=|Q_0|$ denote the number of points in $Q$, let $\amas=\{x_1,\dots,x_n\}$ be a set of $n$ variables,  and denote the points by $Q_0=\{1,\ldots,n \}$, where we agree that the point $i$ corresponds to the variable $x_i$. 
We consider the field ${\cal F} = \mathbb{Q}(x_1,\dots,x_n)$ of rational functions in $x_1,\dots,x_n$, which we call the {\it ambient field}. The {\it cluster algebra} ${\cal A} = {\cal A}(\amas,Q)$ is a $\mathbb{Z}$-subalgebra of ${\cal F}$ defined by a set of generators obtained recursively from $\amas$ in the following manner. Let $i$ be such that $1\leq i \leq n.$ The {\it mutation} $\mu_{x_i,\amas}$ of $(\amas,Q)$ (or $\mu_{x_i}$ or $\mu_i$ for brevity if there is no ambiguity) is defined as follows. Firstly, $Q$ becomes a new quiver $Q^\prime$ obtained from $Q$ by: 
\begin{itemize}
\item[(a)] inserting a new arrow $k\to j$ for each path $k\to i \to j$ of length two with midpoint $i$;
\item[(b)] inverting all arrows of $Q$ passing through $i$;
\item[(c)] deleting each occurrence of a cycle of length two.
\end{itemize}
Secondly, $\amas$ becomes a new set of variables $\amas^\prime=(\amas\setminus\{x_i\}) \cup \{x_i^\prime\}$ where $x_i^\prime\in{\cal F}$ is defined by the so-called {\it exchange relation}: 
\begin{equation*}
x_i x_i^\prime = \prod_{\alpha\in i^+} x_{t(\alpha)} + \prod_{\alpha\in i^-} x_{s(\alpha)}.
\end{equation*}
Let ${\cal X}$ be the union of all possible sets of variables obtained from $\amas$ by successive mutations. Then ${\cal A} = {\cal A}(\amas, Q)$ is the $\mathbb{Z}$-subalgebra of ${\cal F}$ generated by ${\cal X}$. 

Each pair $(\tilde{\amas}, \tilde{Q})$ obtained from $(\amas, Q)$ by successive mutations is called a {\it seed}, and $\tilde{\amas}$ is called a {\it cluster}. The elements $\tilde{x}_1,\dots, \tilde{x}_n$ of a cluster $\tilde{\amas}$ are {\it cluster variables}. Each cluster is a transcendence basis for the ambient field $\mathcal{F}$. The pair $(\amas, Q)$ is the {\it initial seed} and $\amas$ is the {\it initial cluster}.

It has been shown in \cite[Theorem 3]{GSV} that for every seed $(\tilde\amas,\tilde Q)$ the quiver $\tilde Q$ is uniquely defined by the cluster $\tilde \amas$, and we   use the notation  $Q(\tilde\amas)$ for the quiver of the cluster $\tilde\amas$.
More precisely, there is a canonical bijection $p$ from the cluster $\tilde \amas$ to the set of points of the quiver $Q(\tilde\amas)$. We   write $p_x$ for the point in $Q(\tilde\amas)$ corresponding to the cluster variable $x\in \tilde\amas.$

We recall two results from the theory of cluster algebras. The so-called {Laurent phenomenon} is the fact that each cluster variable can be expressed as a Laurent polynomial in the $x_i$, with $i=1,\dots,n$, so that ${\cal A} \subseteq {\mathbb{Z}}[x_1^{\pm 1},\dots, x_n^{\pm 1}]$, see \cite{FZ1}. Also, ${\cal A}$ is {\it of finite type}, that is the set $\cal{X}$ of all cluster variables of $\cal A$ is finite, if and only if there exists a sequence of mutations transforming $Q$ into a Dynkin quiver \cite{FZ2}. In the latter case (as well as in many others), the so-called {\it positivity conjecture} holds true, that is the cluster variables are Laurent polynomials with non-negative coefficients or, equivalently, the cluster variables are contained in ${\mathbb{Z}_{\geq 0}}[x_1^{\pm 1},\dots, x_n^{\pm 1}]$, see \cite{MSW,N,AsRS}.

\subsection{Main definition}

We  define cluster automorphisms as follows.
\begin{definition}
\label{def_main}
Let $\cal A$ be a cluster algebra, 
and let $f:{\cal A} \to {\cal A}$ be an 
  automorphism of ${\mathbb Z}$-algebras. Then $f$   is called a \emph{cluster automorphism} if there exists a seed $(\amas,Q)$ of ${\cal A}$, such that the following conditions are satisfied: 
\begin{itemize}
\item [\textup{(CA1)}] $f(\amas)$ is a cluster; \label{def1} 
\item [\textup{(CA2)}] $f$ is compatible with mutations, that is, for every $x\in \amas$, we have $$f(\mu_{x,\amas}(x)) = \mu_{f(x),f(\amas)}(f(x)).$$ \label{def2}  
\end{itemize}
\end{definition}

\begin{remark}
\label{rmk_initial}
\begin{itemize} \item[\rm(a)] As we shall see in Proposition \ref{prop1} below, if $%
\mathcal{A}$ is a cluster algebra, and $f:\mathcal{A\rightarrow A}$ is an
automorphism of $\mathbb{Z}$-algebras, then $f$ is a cluster automorphism if
and only if it satisfies properties \textup{(CA1)} and \textup{(CA2)} for \emph{every} seed $
\left( \mathbf{x}^{\prime},Q^{\prime }\right) $ of $\mathcal{A}$,
thus it sends clusters to clusters and commutes with any sequence of
mutations.
\item[\rm(b)]
  Every cluster automorphism $f$ is uniquely determined by its value on the initial cluster variables $x_1,\dots, x_n$ and thus extends in a unique way to an automorphism of the ambient field ${\cal F} = \mathbb{Q}(x_1, \dots, x_n)$ by
\begin{equation*}
\frac{p(x_1,\dots,x_n)}{q(x_1,\dots,x_n)} \mapsto \frac{p(f(x_1), \dots, f(x_n))}{q(f(x_1), \dots, f(x_n))},
\end{equation*}
for all polynomials $p,q$.
\end{itemize}
The converse, however, is not true, as shown in the following  Example.
\end{remark}

\begin{example}\label{ex 1}
\label{ex_mutation}
We   give an example of a $\mathbb Z$-automorphism of the ambient field $\cal F$ which does not restrict to a cluster automorphism of $\cal A$. Let $Q$ be the following quiver
\begin{equation*}
\xymatrix{ 1&2\ar[l]&3\ar[l]}
\end{equation*}
and $\amas=(x_1,x_2,x_3)$. Clearly, any change of transcendence basis of ${\cal F} = {\mathbb Q}(x_1,x_2,x_3)$ induces an automorphism of $\cal F$, and such a change is induced, for instance, by a mutation. Let us define $f:{\cal F}\to {\cal F}$ by $f=\mu_{x_1}$, that is 
\begin{eqnarray*}
&&f(x_1) = \mu_{x_1}(x_1)= \frac{1+x_2}{x_1},\\
&&f(x_2) = \mu_{x_1}(x_2) = x_2,\\
&&f(x_3) = \mu_{x_1}(x_3) = x_3. 
\end{eqnarray*}
%
 Then $f(\amas) = \mu_{x_1}(\amas)$ is a cluster. On the other hand, a straightforward calculation gives that 
\begin{equation*}
f\mu_{x_2, \amas}(x_2) = \frac{1+x_2+x_1x_3}{x_1x_2}
\end{equation*}
while
\begin{equation*}
\mu_{f(x_2),f(\amas)}f(x_2) = \frac{x_1+x_3+x_2x_3}{x_1x_2}.
\end{equation*}
Thus condition \textup{(CA2)} is not satisfied and $f$ is not a cluster automorphism of $\cal A$. 

The above automorphism does not even map cluster variables to cluster variables. Indeed, one of the nine cluster variables in the cluster algebra $\mathcal{A}$ has the Laurent polynomial expansion 
\[\frac{x_1 + x_1x_2 + x_3}{x_2x_3}
\]
 and applying $f$ to this cluster variable gives 

%

\[  \frac{1+2x_2+x_2^2 + x_1x_3}{x_1x_2x_3}
\]
which is not a cluster variable in $\mathcal{A}$.
 \end{example}

\subsection{Equivalent characterisations}
Since mutations induce maps of a cluster algebra onto itself, we may wonder why in the previous example we do not obtain a cluster automorphism. As we see below, the reason is that the associated quiver $Q(f(\amas))$
\begin{equation*}
\xymatrix{p_{f(x_1)}\ar[r]&p_{f(x_2)}&p_{f(x_3)}\ar[l]},
\end{equation*}
  is   isomorphic neither to the original quiver $Q$ nor to its opposite $Q^{\textup{op}}$.

When we say that a sequence of mutations $\mu$ transforms a quiver $Q$ into itself, we mean that for any $\alpha \in Q_1$ we have $\mu(s(\alpha))=s(\mu(\alpha))$ and $\mu(t(\alpha))=t(\mu(\alpha))$. When we say that $\mu$ transforms $Q$ into $Q^{op}$, we mean that for any $\alpha\in Q_1$ we have $\mu(s(\alpha))=t(\mu(\alpha))$ and $\mu(t(\alpha))=s(\mu(\alpha))$.

\begin{lemma}\label{lemma1} 
Let $f$ be a $\ZZ$-algebra automorphism of $\mathcal{A}$. Then $f$ is a cluster automorphism if and only if there exists a seed $(\amas,Q)$ such that $f(\amas)$ is a cluster and one of the following two conditions is satisfied:
\begin{itemize}
\item[(a)] there exists an isomorphism of quivers
$\varphi : Q\longrightarrow Q(f(\amas))$ such that $\varphi(p_{x})=p_{f(x)}$ for all $p_{x}\in Q_0$, or
\item[(b)] there exists an isomorphism of quivers
$\varphi : Q^{\textup{op}}\longrightarrow Q(f(\amas))$ such that $\varphi(p_{x})=p_{f(x)}$ for all $p_x\in Q_0$.
\end{itemize}

\end{lemma}

\begin{proof} Consider a cluster $\amas$ in $\mathcal{A}$. For a cluster variable $x_i$ in $\amas$, corresponding, say, to the point $i\in Q_0$, the exchange relation reads: 
\begin{equation*}
\mu_{x_i,\amas}(x_i) = \frac{1}{x_i} \left( \prod_{\alpha \in i^+ }  x_{t(\alpha)} + \prod_{\alpha \in i^-} x_{s(\alpha)}    \right).
\end{equation*}
Since $f$ is an algebra homomorphism, this implies
\begin{equation}
f(\mu_{x_i,\amas}(x_i)) =  \frac{1}{f(x_i)}\left(\prod_{\alpha \in  i^+ }  f(x_{t(\alpha)}) + \prod_{\alpha \in i^-} f(x_{s(\alpha)})    \right).
\label{temp1}
\end{equation}
On the other hand, $f$ induces a map $\varphi :Q_0 \to Q(f(\amas))_0$ defined by $\varphi(p_{x})=p_{f(x)}$.
Hence 
\begin{equation}
\mu_{f(x_i),f(\amas)} (f(x_i)) 
 = \frac{1}{f(x_i)}\left( \prod_{\beta \in \varphi(i)^+ }  f(\amas)_{t(\beta)} + \prod_{ \beta \in \varphi(i)^-}  f(\amas)_{s(\beta)}    \right) ,
\label{temp2}
\end{equation}
where we denote by $f(\amas)_j$ the cluster variable in the cluster $f(\amas)$ which corresponds to 
the point $j\in Q(f(\amas))_0$. Then $f$ is a cluster automorphism if and only if the expressions (\ref{temp1}) and (\ref{temp2}) coincide for every $i$. This is the case if and only if we have one of the following two situations: either
\begin{itemize}
\item [(i)]
\begin{equation*}
\prod_{\alpha \in i^+}  f(x_{t(\alpha)}) = \prod_{ \beta\in \varphi(i)^+ }  f(\amas)_{t(\beta)}  \qquad \mbox{and} \qquad
\prod_{ \alpha\in i^-}  f(x_{s(\alpha)}) = \prod_{ \beta \in \varphi(i)^- }  f(\amas)_{s({\beta)}};
\end{equation*}

or
\item [(ii)]
\begin{equation*}
\prod_{ \alpha \in i^+}  f(x_{t(\alpha)}) =\prod_{ \beta \in \varphi(i)^- }  f(\amas)_{s({\beta)}}  \qquad \mbox{and} \qquad
\prod_{ \alpha\in i^-}  f(x_{s(\alpha)})  = \prod_{ \beta\in \varphi(i)^+ }  f(\amas)_{t(\beta)}.
\end{equation*}
\end{itemize}

Since the set $f\left( \mathbf{x}\right) $ is a transcendance basis of the
ambient field $\mathcal{F}$, this implies that we have either

\begin{itemize}
\item[(i)] $i^{+}=\varphi (i)^{+}$ and $i^{-}=\varphi (i)^{-}$\\ 
or
\item[(ii)] $i^{+}=\varphi (i)^{-}$\bigskip and $i^{-}=\varphi (i)^{+}\\ $
for every $i=1,\dots,n.$
\end{itemize}

Let us now prove that if one of the two situations (i) or (ii) holds for
the point $i$, then the same situation holds for every point of $Q$.

Suppose that we are in the situation (i) for the point $i$, and let $j$ be a neighbour of $i$.
Without loss of generality, we can assume that there is an arrow $\alpha
:j\rightarrow i$ in the quiver $Q$, that is, $\alpha \in i^{-}$. This
implies that there is no arrow from $i$ to $j$ because $Q$ has no cycles of
length two. From the bijection between $i^{-}$ and $\varphi \left(
i^{-}\right) $, we get an arrow $\varphi \left( \alpha \right) :\varphi
\left( j\right) \rightarrow \varphi (i),$ and so $\varphi \left(\alpha \right)
\in \varphi (j)^{+}$. Since $\alpha \in j^{+}$, this implies that there is a
bijection $j^{+}\cong \varphi \left( j\right) ^{+}$. The proof is entirely
similar if we choose $\alpha$ in $i^{+}.$ Proceeding in this way from neighbour
to neighbour, we see that the map $\varphi $ between points extends to an
isomorphism of quivers $Q\cong Q(f(x)).$

Analogously, the situation (ii) yields an isomorphism of quivers $%
Q^{op}\rightarrow Q(f(\mathbf{x}))$.

%
%

 \end{proof}

In the sequel, we shall mostly need quiver isomorphisms satisfying one of
the conditions of Lemma \ref{lemma1}. Accordingly, if $f$ is a cluster automorphism
of the cluster algebra $\mathcal{A}\left( \amas,Q\right) $, then $f$
induces a map (actually a bijection) between the points $\varphi
:Q_{0}\rightarrow Q\left( f\left( \amas \right) \right) _{0}$ by $
\varphi \left( p_{x}\right) =p_{f\left( x\right) }$ for every $x\in 
\amas$. If this bijection $\varphi $ extends to an isomorphism of
quivers $\varphi :Q\rightarrow Q\left( f\left( \amas \right) \right) 
$ then we say that the latter is \emph{induced} by $f$, and that $f$ is a 
\emph{direct cluster automorphism}. Similarly, if $\varphi $ extends to an
isomorphism of quivers $\varphi :Q^{op}\rightarrow Q\left( f\left( 
\amas\right) \right) $ then we also say that $\varphi $ is \emph{induced} by $f$ but then $f$ is called an \emph{inverse cluster
automorphism}.

\begin{proposition}\label{prop1}
Let $f$ be a cluster automorphism. Then $f$ satisfies conditions \textup{(CA1)} and \textup{(CA2)}  for \emph{every} seed.
\end{proposition}

\begin{proof}
Any seed is obtained from the seed $(\amas,Q)$ of Definition \ref{def_main} by a finite sequence of mutations. It is therefore enough to show that if (CA1), (CA2) hold for a seed $(\amas,Q)$, then they hold for any seed $(\amas',Q')$ that is obtained from $(\amas,Q)$ by a single mutation. 
Let $(\amas',Q')$ be such a seed.
Then $(\amas',Q')=\mu_{x,\amas}(\amas,Q)$, for some $x\in \amas$, thus
\[\amas'=\left(\amas \setminus\{x\}\right)\cup\{x'\},\]
with  the exchange relation
 $$x'=\frac{1}{x}\left(\prod_{p_{x}\to p_{x_i} \in Q_1} x_i +\prod_{p_{x}\ot p_{x_j} \in Q_1} x_j \right).$$

It follows that $f(\amas') =\left( f(\amas) \setminus\{f(x)\}\right)\cup\{f(x')\}.$ By (CA1), $f(\amas)$ is a cluster and by (CA2)
\[f(x')=f(\mu_{x,\amas}(x))= \mu_{f(x),f(\amas)}(f(x)).\]
Therefore $f(\amas')=\mu_{f(x),f(\amas)}(f(\amas))$; in particular, $f(\amas')$ is a cluster. This shows (CA1).

Let us show that condition (a) or (b) of Lemma \ref{lemma1} is satisfied for the seed $(\amas',Q')$.
We have $Q'=\mu_x Q$. 
On the other hand,
\[Q(f(\amas'))=Q(f(\mu_{x,\amas}(\amas))) = Q(\mu_{f(x),f(\amas)}f(\amas)) 
= \mu_{f(x)}(Q(f(\amas))),\]
where the second equality follows from the condition (CA2) for the seed $(\amas,Q)$.
Now, one of the conditions (a) or (b) holds for the seed $(\amas,Q)$, thus, if  (a) holds, then there is an isomorphism 
$\varphi:Q\to Q(f(\amas))$ induced by $f$, that is, such that $\varphi(p_{x})=p_{f(x)}$, for every $x\in \amas$, and therefore
\[Q'=\mu_x Q\cong \mu_{f(x)}Q(f(\amas)) = Q(f(\amas')),\]
and this isomorphism sends every point $p_{x_i'}$ in $Q'$ to the point $p_{f(x_i')}$ in $Q(f(\amas'))$. In other words, the condition (a) holds for the seed $(\amas',Q')$.
On the other hand, if condition (b) holds for $(\amas,Q)$, then there is an isomorphism 
$\varphi:Q^{\textup{op}}\to Q(f(\amas))$ induced by $f$, that is, such that $\varphi(p_{x})=p_{f(x)}$, for every $x\in \amas$, and therefore
\[Q'^{\textup{op}}=\mu_x Q^{\textup{op}}\cong \mu_{f(x)}Q(f(\amas)) = Q(f(\amas')),\]
and this isomorphism sends any point $p_{x_i'}$ in $Q'^{\textup{op}}$ to the point $p_{f(x_i')}$ in $Q(f(\amas'))$; thus   condition (b) holds for the seed $(\amas',Q')$. 

Therefore since the structure of $f(Q^\prime)$ coincides with that of $Q^\prime$ or $Q'^{\textup{op}}$, the expressions analogous to (\ref{temp1}) and (\ref{temp2}) for the cluster $\amas'$ are equal, thus (CA2) is satisfied for the
cluster $f\left( \mathbf{x}^{\prime }\right).$ 
\end{proof}

\begin{corollary}
\label{lemma_property}
Let ${\cal A=A}(\amas,Q)$ and  $f:\cal A\to A$ be a cluster automorphism. Then
\begin{itemize}
\item[(a)] If $f$ is direct, then it induces a quiver isomorphism $Q'\cong Q(f(\amas'))$, for any seed $(\amas',Q')$. 
\item[(b)] If $f$ is inverse, then it induces a quiver isomorphism $Q'^{op}\cong Q(f(\amas'))$, for any seed $(\amas',Q')$. 
\end{itemize} \end{corollary}

\begin{proof} This follows from the proof of  Proposition \ref{prop1} and (CA2).  \end{proof}

\begin{remark}
\label{rmk_seeds}
 It follows from (CA1) that a cluster automorphism amounts to a replacement of a cluster by another cluster of the same algebra. Since seeds are uniquely determined  by clusters, a cluster automorphism can equivalently be considered as a ``change of seed". Condition (CA2) says that this change is compatible with mutations so that if $(\amas,Q)$ is a seed of $\cal A$ and $x\in \amas$, then we have the following commutative diagram:
\begin{equation*}
\xymatrix@R35pt@C75pt{  (\amas,Q) \ar[r]^{{f}}\ar[d]^{\mu_{(x,\amas)}} & (f(\amas),Q(f(\amas))\ar[d]^{ \mu_{(f(x),f(\amas))} } \\
(\mu_{x,\amas}(\amas), \mu_{x,\amas}(Q))\ar[r]^{f} & (f\mu_{x,\amas}(\amas), Q(f\mu_{x,\amas}(\amas)) .}
\end{equation*}
\end{remark}

{ We end this subsection with one more characterisation of cluster automorphisms.
\begin{corollary}\label{cor2.7bis}
Let $\mathcal{A}=\mathcal{A}\left(
Q\right)$ and let $f:\mathcal{A}\rightarrow \mathcal{A}$  be
a  $\ZZ$-algebra automorphism. Then $f$ is a cluster automorphism if and only if $f$ maps each cluster to a cluster. 
\end{corollary}
\begin{proof} Assume that for every cluster $ \mathbf{x} $, $f(\mathbf{%
x)}$ is a cluster. We must prove that for every $x\in \mathbf{x}$ we have
a commutative diagram as in the remark above. Let $x^{\prime }$ be the
variable obtained from $x$ by mutation, then $\mathbf{x}^{\prime }=(\mathbf{x}\setminus \left\{ x\right\} )\cup \left\{ x^{\prime
}\right\} $ is a cluster. Because of our hypothesis, 
 $f(\mathbf{x}^{\prime })=(f(\mathbf{x)}
\setminus\left\{ f(x)\right\} )\cup \left\{ f(x^{\prime
})\right\} $ is a cluster as well. On the other hand, mutating in $f(x)$ the
cluster $f(\mathbf{x)}$ yields the cluster $(f(\mathbf{x)}\setminus\left\{ f(x)\right\} )\cup \left\{ y^{\prime }\right\}$. These two clusters are obtained from $f(\mathbf{x)}$ by mutating in the same
variable $f(x)$ therefore $y^{\prime }=f(x^{\prime }).$ 
\end{proof}
}

As a consequence, we see that the notion of direct cluster automorphism
coincides with that of strong automorphism of a cluster algebra \cite{FZ2}, that is, an isomorphism of $\mathbb{Z}$-algebras that maps every seed to an isomorphic seed.
 
\subsection{The group of cluster automorphisms}
Examples and construction techniques for cluster automorphisms are given below. Clearly, the identity on $\cal A$ is a cluster automorphism. In fact, the following lemma holds.

\begin{lemma}
The set $\Aut$ of all cluster automorphisms of $\cal A$ is a group under composition. 
\end{lemma} 

\begin{proof} Let $f,g \in \Aut$. By Remark \ref{rmk_seeds}, a cluster automorphism amounts to replacing the initial seed $\left( \mathbf{x},Q\right) 
$ by another seed whose quiver is isomorphic to either $Q$ or $Q^{op},$
therefore 
 $f^{-1} \in \Aut$ and $gf \in \Aut$.
 \end{proof}
%

\begin{lemma}
\label{lemma_index}
 The set
$\Autp$ of all direct cluster automorphisms of $\cal A$ is a normal subgroup of $\Aut$ of index at most two. 
\end{lemma}

\begin{proof} Clearly, the identity of $\cal A$ is a direct automorphism. Also, if $f,g \in \Autp$, then $fg^{-1} \in \Autp$, therefore $\Autp$ is a subgroup of $\Aut$. The normality follows from the fact that if $f\in\Autp$ and $g\in\Aut$, then $gfg^{-1}$ induces an automorphism of $Q$ even if $g$ induces an anti-isomorphism. 

In order to prove the statement about the index, let us consider a map $\phi: \Aut \to \ZZ_2$ defined by 
\begin{equation}\label{2.3}
\phi(f) = \left\{\begin{array}{l} \bar{0},  \;\mbox{if} \;\; f\in \Autp \\\bar{1}, \;\mbox{if} \;\; f\notin \Autp  \end{array}\right. .
\end{equation}
 
The map $\phi$ is a group homomorphism. Indeed, if $f,g\in\Aut$, then $\phi(fg) = \bar{0}$ if and only if $fg \in \Autp$, that is, if and only if $f$ and $g$ are both direct or both inverse. The latter condition may be written as $\phi(f)=\phi(g)$, which holds if and only if $\phi(f)+\phi(g)=\bar{0}.$ Thus $\phi(fg) = \phi(f)+\phi(g)$, and $\phi$ is a group homomorphism.  Since $\textup{Ker}\, \phi=\Autp$ and ${\rm Im} \,\phi \subseteq \ZZ_2,$ the lemma is proved. 
 \end{proof}
 \begin{example}
Here is an example of an inverse cluster automorphism. Let $Q$ be the following quiver of type $\A_3$: 
\begin{equation*}
\xymatrix{p_{x_1}&p_{x_2}\ar[l]\ar[r]&p_{x_3}}
\end{equation*}
and $\amas=\{x_1,x_2,x_3\}$. The cluster variables computed inside the cluster category ${\cal C}_Q$ (see \cite{BMRRT} or Section \ref{sect_acyclic} below) are as follows: 
\begin{equation*}
\scalebox{0.95}{$\begin{array}{ccccccccccccccc} 
& & x_3 & & & & \frac{1+x_2+x_1x_3}{x_2x_3} & & & & \frac{1+x_2}{x_1}& & & &x_1\\
& \nearrow & & \searrow & & \nearrow & & \searrow& & \nearrow & & \searrow& & \nearrow &  \\
x_2& & & & \frac{1+x_1x_3}{x_2} &&&&\frac{x_2^2 + 2x_2+1+x_1x_3}{x_1x_2x_3}&&&&x_2  \\
& \searrow & & \nearrow & & \searrow & & \nearrow& & \searrow & & \nearrow& & \searrow &  \\
& &x_1& &&&\frac{1+x_2+x_1x_3}{x_1x_2}&&&&\frac{1+x_2}{x_3}&&&&x_3
 \end{array}$}
\end{equation*}
Define a map $f:\cal A\to A$ to be induced by the mutation $\mu_{x_2}$, so that on the initial cluster we have
\begin{equation*}
f(x_1) = x_1, \qquad f(x_2) = \frac{1+x_1x_3}{x_2}, \qquad f(x_3) = x_3.
\end{equation*}
Then $f$ extends to an algebra homomorphism. A straightforward computation gives the images under $f$ of the remaining cluster variables of the algebra: 
\begin{eqnarray*}
&& f \left( \frac{1+x_1x_3}{x_2}  \right) = x_2, \qquad \qquad \qquad f \left( \frac{1+x_1x_3+x_2}{x_2x_3}  \right) = \frac{x_2+1}{x_3}, \\
&& f \left( \frac{1+x_1x_3+x_2}{x_1x_2}  \right) = \frac{x_2+1}{x_1}, \qquad f \left( \frac{x_2^2 + 2x_2+1+x_1x_3}{x_1x_2x_3}  \right) = \frac{x_2^2 + 2x_2+1+x_1x_3}{x_1x_2x_3}, \\
&& f \left( \frac{1+x_2}{x_1}  \right) = \frac{x_2+1+x_1x_3}{x_1x_2}, \qquad f \left( \frac{1+x_2}{x_3}  \right) = \frac{x_2+1+x_1x_3}{x_2x_3} .
\end{eqnarray*}
Thus $f$ is a cluster automorphism sending the   seed $$(\{x_1,x_2,x_3\},
\xymatrix{p_{x_1}&p_{x_2}\ar[l]\ar[r]&p_{x_3}} )$$ to the seed $$
(\{f(x_1),f(x_2),f(x_3)\},\xymatrix{p_{f(x_1)}\ar[r]&p_{f(x_2)}&p_{f(x_3)}\ar[l]})$$
hence $f$ induces an isomorphism of quivers $Q(\amas)^{op}\cong Q(f(\amas))$. 
\end{example}
%

Two quivers $Q$ and $Q'$ are called \emph{mutation equivalent} if there exists a sequence of mutations transforming $Q$ to $Q'$. 

\begin{theorem}\label{thm ref}
Let $\mathcal{ A}=\mathcal{A}(\amas,Q)$.
\begin{itemize}
\item[(a)] {If $Q$ and $Q^{op}$ are mutation equivalent then the index of $\aut$ in $\Aut$ is two.}
\item[(b)] If $Q$ and $Q^{op}$ are not mutation equivalent then $\Aut=\Autp$.
\end{itemize}
\end{theorem}
 
\begin{proof}{
If $Q$ and $Q^{op}$ are mutation equivalent then there exists a sequence of mutations $\mu$ such that 
$\mu(\amas,Q(\amas))=(\amas',Q(\amas'))$ together with two isomorphisms of quivers $\varphi:Q(\amas)\stackrel{\cong }{\to}Q$ and $\varphi':Q^{op}\stackrel{\cong }{\to}Q(\amas')$. 
\ralf{Observe that there is no reason for the isomorphisms $\varphi$ and $\varphi'$ to be compatible with the canonical bijection between the points $Q_0\cong Q(\amas')_0$.}
Define a map $f:\amas\to\amas'$ as the following composition of the bijections on points 
\[Q(\amas)_0\stackrel{\varphi}{\longto} Q_0 \stackrel{=}{\longto} (Q^{op})_0\stackrel{\varphi'}{\longto} Q(\amas')_0.\] Since $\amas$ and $\amas'$ are transcendence bases of the ambient field, then $f$ extends to a $\ZZ$-algebra automorphism of $\mathcal{A}$, and, moreover, $\varphi' :Q^{op}\stackrel{\cong }{\to}Q(\amas')$  is an isomorphism of quivers that satisfies $\varphi'(p_x)=p_{f(x)}$,
and it follows from Lemma 2.3 that  $f\in \Aut\setminus\Autp$. 
This shows (a).}

In order to show (b), suppose that  $Q$ and $Q^{op}$ are not mutation equivalent. Suppose that there exists $f\in\Aut\setminus\Autp$. Then there exists a seed $(\amas,Q(\amas))$ in $\mathcal{A}$ with $Q(\amas)\cong Q$ whose image  $(f(\amas),Q(f(\amas)))$ under $f$ is a seed in $\mathcal{A}$ whose quiver is isomorphic to $Q^{op}$. Thus $Q$ and $Q^{op}$ are mutation equivalent, a contradiction.
\end{proof}

\begin{corollary}\label{cor two new}
{If there exists $\sigma\in\Aut\setminus\aut$ of order two, then $\Aut=\aut\rtimes \ZZ_2$.}
\end{corollary}
\begin{proof}
\ralf{If $\Aut= \Autp $ then there is an exact sequence of groups. }
\[ 1 \to \Autp \to \Aut \stackrel{\phi}{\to} \ZZ_2\to1\]
\ralf{where $\phi$ is defined in equation (\ref{2.3}). We have 
$\Aut\cong \Autp \rtimes\ZZ_2$ if and only 
if this sequence splits, and this is the case if and only if there exists an inverse automorphism $\sigma$ of order 2. }
%

\end{proof}

\begin{example}\label{ex torus}
Let $Q$ be the quiver 
$$\xymatrix@R20pt@C25pt{1\ar@<2pt>[rd]\ar@<-2pt>[rd] &&2 \ar@<2pt>[ll]\ar@<-2pt>[ll]\\
&3\ar@<2pt>[ru]\ar@<-2pt>[ru]
}$$
and $\mathcal{A}=\mathcal{A}(\{x_1,x_2,x_3\},Q\})$ its cluster algebra. Then the mutation  of $Q$ in $1$ is the quiver
$$ \xymatrix@R20pt@C25pt{1\ar@<2pt>[rr]\ar@<-2pt>[rr] &&2 \ar@<2pt>[ld]\ar@<-2pt>[ld]\\
&3\ar@<2pt>[lu]\ar@<-2pt>[lu]
},$$ and
thus it induces an inverse cluster automorphism $f$ given by  $$f(x_1)=\frac{x_2^2+x_3^2}{x_1}, \ f(x_2)=x_2 \textup{ and } f(x_3)=x_3.$$

Mutating once more, this time in $x_2$ yields back the quiver $Q$, and thus  it induces a direct cluster automorphism $g$ given by 
$$g(x_1)=\frac{x_2^2+x_3^2}{x_1}, \ g(x_2)=\frac{1}{x_2} \left(\left(\frac{x_2^2+x_3^2}{x_1}\right)^2+x_3^2\right)\textup{ and } g(x_3)=x_3 .$$

Note that in this particular example, any sequence of mutations will either produce the quiver $Q$ (if the number of mutations is even) or its opposite (if the number of mutations is odd), and thus, in this example, any sequence of mutations induces a cluster automorphism.

\end{example}

\begin{example}\label{ex ref}
Let $\mathcal{A}=\mathcal{A}(\amas,Q)$, where  $Q$ is the quiver \[\xymatrix{&2\ar@<3pt>[rd]\ar@<0pt>[rd]\\1\ar@<3pt>[rr]\ar@<-3pt>[rr]\ar@<0pt>[rr] \ar[ur] &&3\ .}\]
Then $Q $ is not mutation equivalent to $Q^{op}$ (and therefore  $\aut=\Aut$). This can be seen using the associated cluster category, see \cite{BMRRT} or Section 3 below. Indeed, assume that $Q$ is mutation equivalent to $Q^{op}$. There exists a local slice $\Sigma$ (in the sense of \cite{ABS2}) in the cluster category ${\cal C}_Q$ of $Q$ whose quiver is isomorphic to $Q$. Since $Q$ and $Q^{op}$ are acyclic, the slice $\Sigma$ lies in the transjective component $\Gamma_{tr}$ of the Auslander-Reiten quiver of ${\cal C}_Q$, and this transjective component is of the form $\ZZ Q^{op}$. In particular, the point (corresponding to) $3$ in $\Sigma$ is the target of five arrows in $\Gamma_{tr}$. The point $1$ in $\Sigma$ is the source of three of the arrows of target $3$, and the point $2$ is the source of the remaining two arrows. Since $\Gamma_{tr}\cong \ZZ Q^{op}$, this implies the existence of an arrow from $2$ to $1$, and then the quiver of $\Sigma$ is not isomorphic to $Q$, a !
 
 contradiction. 


\end{example}

\subsection{Cluster automorphisms induced by quiver automorphisms}
 We now show how any automorphism of the quiver $Q$ induces a direct cluster automorphism of ${\cal A} = {\cal A}(\amas,Q)$. Let 
$\sigma \in {\rm Aut}\, Q$. Define a map $f_\sigma: \amas \to \amas$  by
$f_\sigma(x)= x'$, for $x \in \amas$, where $x'\in\amas $ is the unique cluster variable such that $\sigma(p_x)=p_{x'}$.

 Then $f_\sigma$ permutes the cluster variables in $\amas$ and clearly extends to a unique automorphism $f_\sigma: \cal F \to F$ of the ambient field. We now show that $f_\sigma$ is a direct cluster automorphism of $\cal A.$ 

\begin{proposition}\label{prop 2.13}
\label{prop_kernel}
The map $F:\sigma \mapsto f_\sigma$ is a group homomorphism from ${\rm Aut}\, Q$ to $\Autp$ whose kernel is given by the stabiliser ${\rm Stab} \, Q_0$ of the points of $Q.$
\end{proposition} 

\begin{proof} In order to show that $f_\sigma$ is a cluster automorphism we must prove that conditions (CA1) and (CA2) are satisfied for the initial cluster $\amas$. Since $f_\sigma(\amas) = \amas$, the first condition obviously holds.

 Let now $x \in \amas$ and consider the mutation $\mu_x = \mu_{x,\amas}$. Since $\sigma \in {\rm Aut}\, Q$, we have $\sigma Q = Q$ and hence $\mu_{f_\sigma(x),\amas}$ is a mutation of the seed $(\amas, Q)$. 

We want to show (CA2), that is
\[ f_\sigma(\mu_{x,\amas}(x))=\mu_{f_\sigma(x),\amas}(f_\sigma(x)).\]
Using the exchange relations, we have 
\begin{equation} \label{eq25}
f_\sigma(\mu_{x,\amas}(x)) 
= \frac{1}{f_{\sigma}(x)} \left(\,\prod_{\za\in p_x^+}f_\sigma(x_{t(\alpha)}) + \prod_{\za\in p_x^-}f_\sigma(x_{s(\alpha)})\right) ,
\end{equation}
while 
\begin{equation} \label{eq26}
\mu_{f_\sigma(x),\amas}(f_\sigma(x)) 
= \frac{1}{f_\sigma(x)} \left(\prod_{\zb\in \sigma(p_x)^+} x_{t(\zb)} + \prod_{\zb\in \sigma(p_x)^-}x_{s(\zb)}\right) .
\end{equation}
Since $\sigma  $ is an automorphism of $Q$, we have $\sigma(t(\za))=t(\sigma(\za))$, $\sigma(s(\za))=s(\sigma(\za))$, and $\sigma(p_x^+)=\sigma(p_x)^+$, $\sigma(p_x^-)=\sigma(p_x)^-$.
Therefore
\begin{eqnarray*}
\prod_{\za\in p_x^+}f_\sigma(x_{t(\alpha)}) = \prod_{\za\in p_x^+}x_{t(\sigma(\alpha))} =\prod_{\zb\in \sigma(p_x)^+} x_{t(\zb)} && \textup{and} 
\\
\prod_{\za\in p_x^-}f_\sigma(x_{s(\alpha)}) = \prod_{\za\in p_x^-}x_{s(\sigma(\alpha)) }=\prod_{\zb\in \sigma(p_x)^-} x_{s(\zb)},
\end{eqnarray*}
which shows that the right hand sides of  equations (\ref{eq25}) and (\ref{eq26}) are equal, and hence (CA2). This completes the proof that $f_\sigma$ is a cluster automorphism. It is direct because $\sigma Q \cong Q$. 
This shows that the map $F$ is well-defined. It is easy to see that $F$ is a group homomorphism with   kernel ${\rm Stab} \, Q_0.$

 \end{proof}

\begin{example}
The automorphism of the Kronecker quiver 
\begin{equation*} \xymatrix{1&2\ar@<2pt>[l]\ar@<-2pt>[l]}
\end{equation*}
which fixes the points and interchanges the arrows lies in the kernel of $F$.
\end{example}

We finally observe that an anti-automorphism of the quiver induces in
the same way an inverse cluster automorphism. Let $\mathcal{A}\left( \mathbf{x},Q\right) $ be a cluster algebra and $\sigma $ be an anti-automorphism of $
Q$. We define $f_{\sigma }:\mathbf{x}\rightarrow \mathbf{x}$ by setting $
f_{\sigma }\left( x\right) =x^{\prime }$ where $x^{\prime }$ is the unique
cluster variable such that $\sigma \left( p_{x}\right) =p_{x^{\prime }}.$ Then $f_{\sigma }$ is clearly an automorphism of the ambient field $\mathcal{
F}$.
\begin{proposition}\label{prop anti}  With the above notation, the map $
f_{\sigma }$ is an inverse cluster automorphism of $\mathcal{A}
\left( \mathbf{x},Q\right) $.
\end{proposition}
\begin{proof}
The proof is entirely similar to that of Proposition \ref{prop 2.13} and will be
omitted.
\end{proof}

\begin{corollary}\label{cor anti new} {Let $\mathcal{A}$ be a cluster algebra with a seed $(\amas,Q)$ such that $Q$  admits an anti-automorphism.
Then $\Aut=\Autp\rtimes\ZZ_2$.}
\end{corollary} 
\begin{proof}
{In this situation,  Proposition 2.15 yields that there exists $\sigma\in\Aut\setminus\aut$ of order two, and the result follows from Corollary \ref{cor two new}.}
\end{proof}
%


\begin{example}\label{ex anti atilde}
{The cluster algebras of euclidean type $\tilde{\mathbb{A}}$ admit an inverse cluster automorphism. Indeed, for any cluster algebra of type $\tilde{\mathbb{A}}$, there exist integers $p,q$ such that the cluster algebra
has a seed whose quiver is of the following form.}
\[\xymatrix{&2\ar[r]&3\ar[r]&\cdots\ar[r]&p\ar[rd]\\
1\ar[ur]\ar[dr]&&&&&p+q\\
&p+1\ar[r]&p+2\ar[r]&\cdots\ar[r]&p+q-1\ar[ur]
}
\]
{ This quiver has an anti-automorphism given on the points by the permutation that exchanges 1 with $p+q$, $\ell $ with $p+2-\ell$ for $\ell=2,3,\cdots,p$,  and $p+m$ with $p+q-m$ for $m=1,2,\ldots,q-1$. Therefore Proposition \ref{prop anti} implies that the cluster algebra has an inverse cluster automorphism. }
\end{example}

\section{The acyclic case}
\label{sect_acyclic}


\subsection{The cluster category} In this section, we assume that $Q$ is an acyclic quiver. In this case, the combinatorics of cluster variables are encoded in the cluster category. Let $k$ be an algebraically closed field, $kQ$ the path algebra of $Q$ and ${\rm mod} \, kQ$ the category of finitely generated right $kQ$-modules. We denote by ${\rm ind}\, kQ$ a full subcategory of ${\rm mod} \, kQ$ consisting of exactly one object from each isomorphism class of indecomposable   $kQ$-modules. For $x \in Q_0$, we denote by $P_x$ the corresponding indecomposable projective $kQ$-module. For properties of ${\rm mod}\, kQ$ and its Auslander-Reiten quiver $\Gamma ({\rm mod}\, kQ)$, we  refer the reader to \cite{ARS,ASS}.

We denote by ${\cal D}^b({\rm mod}\,kQ)$ the bounded derived category over ${\rm mod} \, kQ$. This is a triangulated Krull-Schmidt category having Serre duality and hence almost split triangles. Since $kQ$ is hereditary, the Auslander-Reiten quiver $\Gamma({\cal D}^b({\rm mod}\,kQ))$ of ${\cal D}^b({\rm mod}\,kQ)$ is well-understood \cite{H}. The cluster category is defined to be the orbit category of ${\cal D}^b({\rm mod}\, kQ)$ under the action of the automorphism $\tau^{-1}[1]$, where $\tau$ is the Auslander-Reiten translation and $[1]$ is the shift of ${\cal D}^b({\rm mod}\, kQ)$, see \cite{BMRRT}. Then ${\cal C}_Q$ is also a triangulated Krull-Schmidt category having almost split triangles, and the projection functor ${\cal D}^b({\rm mod} \,kQ) \to {\cal C}_Q$ is a functor of triangulated categories commuting with the Auslander-Reiten translation \cite{K}. Moreover, ${\cal C}_Q$ is a $2$-Calabi-Yau category \cite{BMRRT}. Let ${\rm ind} \,{\cal C}_Q$ denote a full subcat!
 
 egory of ${\cal C}_Q$ consisting of exactly one object from each isomorphism class of indecomposable objects in ${\cal C}_Q$, then ${\rm ind}\, {\cal C}_Q$ can be identified with the disjoint union of ${\rm ind} \,kQ$ and $kQ[1] = \{P_x[1] | x\in Q_0\}$, the shifts of the indecomposable projective $kQ$-modules. We always use this identification in the sequel. The Auslander-Reiten quiver $\Gamma({\cal C}_Q)$ of ${\cal C}_Q$ is the quotient of $\Gamma({\cal D}^b({\rm mod}\,kQ))$ under the action of the quiver automorphism $\tau^{-1}[1]$. This Auslander-Reiten quiver has always a unique component containing all the objects of $kQ[1]$. This is the \emph{transjective} component of $\Gamma({\cal C}_Q)$ and is denoted by $\Gamma_{tr}$. If $Q$ is a Dynkin quiver, then $\Gamma({\cal C}_Q) \cong \Gamma_{tr}. $ Otherwise, $\Gamma_{tr}$ is isomorphic to the repetition quiver $\Gamma_{tr} \cong {\mathbb Z}Q$ of $Q$ (see \cite{ASS}), and there are infinitely many so-called \emph{regular}!
 
  components which are either stable tubes (if $Q$ is euclidean) or of 
type $\mathbb{ZA}_\infty$ (if $Q$ is wild). 

Let $n = |Q_0|$. There exists a map
\begin{equation*}
X_? : ({\cal C}_Q)_0 \to \mathbb Z [x_1^{\pm 1}, \dots, x_n^{\pm 1}]
\end{equation*}
called the \emph{canonical cluster character}, or the \emph{Caldero-Chapoton map}. The map $X_?$ induces a bijection between the $M$ in ${\rm ind}\, {\cal C}_Q$ which have no self-extensions, and the cluster variables $X_M$, see \cite{CK06}. Under this bijection, the clusters correspond to the so-called {\it tilting} objects (also known as cluster-tilting objects) in ${\cal C}_Q$. In practice, the map $X_?$ is difficult to compute explicitly. An easier method for computing the cluster variables is via the frieze functions \cite{AsRS,AD,ADSS}. 

\subsection{Cluster automorphisms in the acyclic case}  In this section, we prove that if $Q$ is acyclic, then  the cluster automorphisms of ${\cal A=A}(\amas, Q)$ are entirely determined by the quiver automorphisms of the transjective component $\Gamma_{tr}$.

\begin{lemma}\label{lemma 3.1}
Let $f$ be a cluster automorphism of ${\cal A}={\cal A}(\amas,Q)$, where $Q$ is acyclic.
\begin{itemize}
\item[(a)] If $f$ is direct, then it induces a triangle equivalence $f_{{\cal D}}: {\cal D}^b({\rm mod} \; kQ) \to {\cal D}^b({\rm mod} \; kQ) $.
\item[(b)] If $f$ is inverse, then it induces a triangle equivalence $f_{{\cal D}}: {\cal D}^b({\rm mod} \; kQ) \to {\cal D}^b({\rm mod} \; kQ^{op}) $.
\end{itemize}
\end{lemma}

\begin{proof} 
Let  $X_?$ denote as before the canonical cluster character. For each $x$ in $\mathbf{x}
$, there exists a unique indecomposable object $M_{x}$ in the cluster
category $\mathcal{C}_{Q}$ such that $f(x)=X_{M_{x}}$. Because $\mathbf{x}$
is a cluster, then $M=\oplus _{x\in \mathbf{x}}M_{x}$ is a tilting object in 
$\mathcal{C}_{Q}$.
\ralf{ If $f$ is direct then $\textup{End}\, M\cong kQ$.}
 Therefore, the set $\left\{
M_{x}\mid x\in \mathbf{x}\right\} $ forms a local slice in $
\mathcal{C}_{Q}$. Due to \cite{ABS2}, there exists a slice $\Sigma $ in a
transjective component of ${\cal D}^{b}\left( \textup{mod}\,kQ\right) $, isomorphic to $%
Q$, such that, for each $x\in \mathbf{x}$, the object $M_{x}$ lifts to $%
\widetilde{M_{x}}$ in $\Sigma$. Because $\widetilde{M}=\oplus _{x\in 
\mathbf{x}}\widetilde{M_{x}}$ is a slice complex in ${\cal D}^{b}\left( \textup{mod}\,
kQ\right)$, it is also a tilting complex. Therefore the triangle functor $
f_{\mathcal{D}}=-\otimes _{kQ}^{\mathbb{L}}\widetilde{M}[-1]:\mathcal{D}^{b}\left( \textup{mod}\,kQ\right) \rightarrow \mathcal{D}^{b}\left( \textup{mod}\,kQ\right) $ is a triangle
equivalence. 
\ralf{This shows (a). 
If $f$ is inverse then   $\textup{End}\, M\cong kQ^{op}$, and again the set $\left\{
M_{x}\mid x\in \mathbf{x}\right\} $ forms a local slice in $
\mathcal{C}_{Q}$.  Again due to \cite{ABS2}, there exists a slice $\Sigma $ in a
transjective component of ${\cal D}^{b}\left( \textup{mod}\,kQ\right) $, isomorphic to $%
Q^{op}$, such that, for each $x\in \mathbf{x}$, the object $M_{x}$ lifts to $%
\widetilde{M_{x}}$ in $\Sigma$. Because $\widetilde{M}=\oplus _{x\in 
\mathbf{x}}\widetilde{M_{x}}$ is a slice complex in ${\cal D}^{b}\left( \textup{mod}\,
kQ\right)$, it is also a tilting complex. Again the triangle functor $
f_{\mathcal{D}}=-\otimes _{kQ}^{\mathbb{L}}\widetilde{M}[-1]:\mathcal{D}^{b}\left( \textup{mod}\,kQ\right) \rightarrow \mathcal{D}^{b}\left( \textup{mod}\,kQ\right) $ is a triangle
equivalence. 
}
%
%
 \end{proof}


Recall  that a morphism of translation quivers is a morphism of quivers which commutes with the translation. 
 
\begin{corollary}
Let $f$ be a cluster automorphism of ${\cal A=A}(\amas,Q)$, where $Q$ is acyclic.
\begin{itemize}
\item[(a)] If $f$ is direct, then it induces a quiver automorphism of the transjective component $\Gamma_{tr}$ of  $ \Gamma({\cal C}_Q)$.
\item[(b)] If $f$ is inverse, then it induces a quiver anti-automorphism of the transjective component $\Gamma_{tr}$ of  $ \Gamma({\cal C}_Q)$.
\end{itemize}
\end{corollary}
\begin{proof}
We only prove (a), because the proof of (b) is similar. Let $f$ be
direct. As seen in Lemma \ref{lemma 3.1}, it induces a triangle equivalence $f_{\mathcal{
D}}:\mathcal{D}^{b}\left( \textup{mod}\,kQ\right) \rightarrow\mathcal{D}^{b}\left( \textup{mod}\,
kQ\right) $ mapping the slice $\mathcal{P}=\left\{ P_{x}\mid
x\in \mathbf{x}\right\} $ consisting of the indecomposable projective 
$kQ$-modules to the isomorphic slice 
$\Sigma =\left\{ \widetilde{M_{x}}\mid x\in \mathbf{x}\right\}$, 
and both slices may be assumed to lie in
the same transjective component $\Gamma $ of $\Gamma \left( \mathcal{D}^{b}\left( 
\textup{mod}\,kQ\right) \right)$. Clearly, $f_{\mathcal{D}}\left( P_{x}\right) =%
\widetilde{M_{x}}$, and moreover $f_{\mathcal{D}}$ induces a quiver
isomorphism $\mathcal{P}\cong \Sigma ,$ which extends uniquely to a
quiver automorphism of $\Gamma$.
\ralf{On the other hand, $f_{\mathcal{D}}$ 
is a triangle equivalence, hence commutes with the shift of $ \mathcal{D}^{b}\left( 
\textup{mod}\,kQ\right)$ and the 
Auslander-Reiten translation of $\Gamma$. Therefore it induces an automorphism of the 
translation quiver $\Gamma_{tr}$. }
%
%
\end{proof}

\begin{remark}\label{rem 3.3} As a direct consequence of the above proofs, if $f$
is nontrivial, then it acts nontrivially on the transjective component.
\end{remark}

\begin{example}
{{To illustrate Remark \ref{rem 3.3} we give an example of a cluster automorphism defined by permuting two regular components of the cluster category and show that this automorphism induces a nontrivial action on the transjective component. }

{ Let $(\amas,Q)$ be the seed with cluster $\amas=\{x_1,x_2,x_3,x_4\}$ and quiver  $$\xymatrix@R=10pt{&2\ar[ld]\\1\ar@<2pt>[rr]\ar@<-2pt>[rr] &&4\ar[ul]\ar[dl]\ . \\ &3\ar[lu]}$$
This is a seed of type $\tilde {\mathbb{A}}_{2,2}$: indeed, mutating the seed in $x_2$ and then in $x_3$ one gets a new seed with cluster $\amas'=\{x_1,\frac{x_1+x_4}{x_2},\frac{x_1+x_4}{x_3},x_4\}$ and quiver  $$\xymatrix@R=10pt{&2\ar[rd]\\1\ar[ru]\ar[rd] &&4\ . \\ &3\ar[ru]}$$
}

{In the corresponding cluster category, each of the objects associated with $x_2$ and $x_3$ sits in the mouth of a tube of rank 2. Therefore permuting the two tubes induces a cluster automorphism given on the cluster $\amas$ by the permutation of the two variables $x_2 $ and $x_3$. 
On the cluster $\amas'$ this automorphism acts by permutation of the variables $\frac{x_1+x_4}{x_2}$ and $\frac{x_1+x_4}{x_3}$. Therefore the induced action on the transjective component is nontrivial.}}
\end{example}

 We now relate cluster automorphisms to automorphisms of the original quiver. We call a mutation an \emph{APR-mutation} if it is applied to a source or to a sink. The letters APR stand for Auslander, Platzeck and Reiten and evoke the similarity between such mutations and the so-called APR-tilts \cite{APR}.  Equivalently, one may think of
APR-mutations as a generalisation of the reflection functors of Bernstein,
Gelfand and Ponomarev \cite{BGP}. We also note that, given an APR-mutation, the
new cluster variable is obtained from the old ones by using the frieze
function.

For our purposes, it is useful to understand how an APR-mutation translates
into terms of the cluster category. Let $\mathcal{A}\left( \mathbf{x},Q\right) 
$ be a cluster algebra, with $Q$ acyclic, and $\mu _{x}$ be an APR-mutation
corresponding to a source (say) in $Q$. Then $\mu _{x}$ maps $\mathcal{A}%
\left( \mathbf{x},Q\right) $ to $\mathcal{A}\left( \mu _{x}\mathbf{x},\mu
_{x}Q\right)$. Identifying $Q$ with the full subquiver $kQ[1]=\left\{
P_{x}[1]\mid x\in \mathbf{x}\right\} $ of the transjective
component $\Gamma _{tr}$ of $\Gamma \left( \mathcal{C}_{Q}\right) $, the
application of $\mu _{x}$ to $kQ[1]$ clearly amounts to replacing $kQ[1]$ by
the new tilting object $(kQ[1]\setminus \left\{ P_{x}[1]\right\} )\cup
\left\{ P_{x}\right\} $ whose quiver is also a slice in $\Gamma _{tr}$ ,
though not isomorphic to $Q$.

 \begin{lemma}
\label{lemma_APR}
Let $f$ be a cluster automorphism of ${\cal A} (\amas,Q)$, where $Q$ is acyclic. Then there exists a sequence $\mu$ of APR mutations  such that $\mu(\amas)=f(\amas)$ \ralf{as sets}.
\end{lemma}
 
\begin{proof} Let $\amas=\{x_1,\ldots,x_n\}$. If $f\in \Aut$, then  $(f(\amas), Q(f(\amas)) $ is a seed such that $Q(f(\amas))$ is isomorphic to $Q$ or $Q^{\textup{op}}$.
 In particular,  $Q(f(\amas))$  is acyclic. 
 Therefore, the cluster $f(\amas)$ corresponds to a tilting object $T=\oplus_{i=1}^n T_i$ in ${\cal C}_Q$ such that the indecomposable summand $T_i$ corresponds to the variable $f(x_i)$ for all $i$ such that $1\le i\le n$.
 Then  $Q(f(\amas))$ is the ordinary quiver of the cluster-tilted algebra ${\rm End_{{\cal C}_Q}}T$, and since  $Q(f(\amas))$ is acyclic, it follows that ${\rm End_{{\cal C}_Q}}T$ is hereditary. Therefore, $T$ is a local slice in the transjective component $\Gamma_{tr}$ of $\Gamma({\cal C}_Q)$, see \cite[Corollary 20]{ABS2}, and hence there exists a sequence $\mu$ of APR-mutations such that $\mu(\amas)=f(\amas)$ \ralf{as sets. Observe that if one lifts the slice $(kQ[1]\setminus \{P_x [1]\}) \cup \{P_x \}$ 
to a slice $\Sigma$, say, in the derived category, then the endomorphism algebra of $\Sigma$ 
is obtained from $kQ$ by an APR-tilt (see \cite{APR}).} \end{proof}

\begin{remark} Lemma \ref{lemma_APR} states that the action of a cluster automorphism on a cluster is given by a composition of APR-mutations.
However, it is not true in general that the automorphism $f$ itself is given by a sequence of APR-mutations, as we show in the following example.
\end{remark}
\begin{example}\label{ex atilde}
Let $Q$ be the quiver $
\xymatrix{p_{x_1}\ar@/^10pt/[rr]\ar[r]&p_{x_2}\ar[r]&p_{x_3}}$. The transjective component $\zG_{tr}$ of the Auslander-Reiten is the infinite quiver illustrated in Figure \ref{fig atilde}. The position of the initial cluster $\{x_1,x_2,x_3\}$ as well as one other variable $u$ are depicted in the figure. 
There is a direct cluster automorphism $f\in\aut$ defined by $f(x_1)=x_2, f(x_2)=x_3$ and $f(x_3)=u$. Note that $f$ induces an isomorphism of quivers $Q(\amas)\cong Q(f(\amas))$.
The corresponding APR-mutation $\mu$ of Lemma \ref{lemma_APR} is the mutation in $x_1$. We have $\mu_{x_1}(\amas)=\{u,x_2,x_3\}=f(\amas)$. 
But the mutation $\mu$ fixes the variables $x_2,x_3$ and sends $x_1$ to $u$, which shows that the automorphism $f$ is not given by $\mu$. 
Also note that $\mu$ sends the sink $p_{x_3}$ in the quiver $Q$ to a point $p_{x_3}$ of the quiver $Q(f(\amas))$ which is neither a source nor a sink, whence  $\mu$ does not induce a quiver isomorphism between $Q$ and $Q(f(\amas))$ or between $Q^{\textup{op}}$ and $Q(f(\amas))$; consequently, $\mu$ does not induce a cluster automorphism. 
\begin{figure}
\[ 
\xymatrix@R=10pt@C=10pt{
&&&\cdot\ar[rrd]\ar@/_5pt/[rdd]&&&x_3\ar[rrd]\ar@/_5pt/[rdd]&&&\cdot\ar[rrd]\ar@/_5pt/[rdd]&&&\cdot\ar@/_5pt/[rdd]&\\
\cdots&&\cdot\ar[ru]\ar[rrd]&&&x_2\ar[ru]\ar[rrd]&&&\cdot\ar[ru]\ar[rrd]&&&\cdot\ar[ru]\ar[rrd]&&\cdots\\
&\cdot\ar[ru]\ar@/^12pt/[rruu]&&&x_1\ar[ru]\ar@/^12pt/[rruu]&&&u\ar[ru]\ar@/^12pt/[rruu]&&&\cdot\ar[ru]\ar@/^12pt/[rruu]&&&
}
\]
\caption{ $\zG_{tr}$ in Example \ref{ex atilde}} \label{fig atilde}
\end{figure}
\end{example}

We have seen in Examples \ref{ex_mutation} and \ref{ex atilde} that not every sequence of
APR-mutations is a cluster automorphism. As follows from Lemma \ref{lemma1}, if a
sequence of mutations $\mu$ transforms $Q$ into itself or its opposite $Q^{op}$ in such a way that the bijection $\mu: Q_0 \to \mu(Q)_0$ extends to
an isomorphism of quivers of the form either $\mu: Q \to \mu(Q)$ or $\mu:
Q^{op} \to \mu(Q)$, then $\mu$ induces a cluster automorphism $f_{\mu }\in 
\textup{Aut}\, \mathcal{A}\left( \mathbf{x},Q\right) $ defined on the
initial cluster by $f_{\mu }\left( x_{i}\right) =\mu \left( x_{i}\right) .$
Note that this applies also to quivers which are not
acyclic.

Now we are in a position to prove the first main theorem of this section.

\begin{theorem}
\label{theorem_main}
The map $\phi: {\rm Aut}(\Gamma_{tr}) \to \Autp$ defined  by $\phi(g)(x_i) = X_{g(P_i[1])}$ for $g\in {\rm Aut}(\Gamma_{tr}) $ and $1\leq i\leq n,$ is a surjective group homomorphism, whose kernel equals the stabiliser ${\rm Stab}(\Gamma_{tr})_0$ of the points in $\Gamma_{tr}$. 
\end{theorem}

\begin{proof} (a) First we prove that the map $\phi$ is well defined. Let $f=\phi(g)$. Now, since $g\in {\rm Aut} (\Gamma_{tr})$, then the local slice $g(kQ[1])$ of $\Gamma_{tr}$ has underlying quiver isomorphic to $Q$. Therefore $(f(\amas),Q(f(\amas))$  is a seed with  quiver isomorphic to $Q$. This shows that $f$ satisfies the condition (CA1). 
Moreover, $f$ satisfies the condition (CA2) because $g$ is an automorphism of $\zG_{tr}$, and the canonical cluster character map commutes with mutations.
This proves that $f$ is a cluster automorphism.
Since $g$ is an automorphism of $\zG_{tr}$, and so preserves the orientation of the arrows between the $P_i[1]$, then $f$ is a direct automorphism.

(b) We now prove that $\phi$ is a group homomorphism. Let $g_1,\, g_2 \in {\rm Aut}\, \Gamma_{tr}$, then the following equalities hold:
\begin{eqnarray*}
\phi(g_1 g_2)(x_i) = X_{g_1g_2(P_i[1])}, \\
\phi(g_1)\phi(g_2)(x_i) = \phi(g_1)(X_{g_2(P_i[1])})
\end{eqnarray*}
with $1\leq i\leq n.$ We need to establish the equality of these two expressions. Note that $X_{g_2(P_i[1])}$ is a Laurent polynomial which we denote by $L_{2,i}(x_1,\dots,x_n)$ and, similarly, $X_{g_1(P_i[1])}$ is a Laurent polynomial which we denote by $L_{1,i}(x_1,\dots,x_n)$. 

It follows directly from the definition of a morphism of translation quivers that $g_1g_2kQ[1]$ is a local slice with quiver isomorphic to $Q$. Thus, denoting by $g_2(i)$ the image in $g_2kQ[1]$ of the point $P_i[1]$, we get
\begin{equation*}
X_{g_1g_2(P_i[1])} = L_{1,g_2(i)}(  L_{2,1}(x_1,\dots,x_n)     ,\dots, L_{2,n}(x_1,\dots,x_n)).
\end{equation*}
On the other hand, $g_2kQ[1]$ is also a local slice with quiver isomorphic to $Q$, therefore
\begin{equation*}
\phi(g_1)X_{g_2(P_i[1])} = L_{1,g_2(i)}(  L_{2,1}(x_1,\dots,x_n)     ,\dots, L_{2,n}(x_1,\dots,x_n)),
\end{equation*}
which completes the proof. 

(c) We now prove that the kernel of $\phi$ equals ${\rm Stab}(\Gamma_{tr})_0$. Assume that $g, \,g^\prime \in {\rm Aut}\,\Gamma_{tr}$  are such that $\phi(g)=\phi(g^\prime)$. By definition of $\phi$, this means that $g(P_i[1]) = g^\prime(P_i[1])$, for every $i$. Thus $g$ and $g^\prime$ coincide on the initial slice. Therefore $g(M)=g^\prime(M)$ for every indecomposable object $M$ of ${\cal C}_Q$, that is, $g^{-1}g^\prime \in {\rm Stab}(\Gamma_{tr})_0$. 

(d) To show that $\phi$ is surjective, let $f\in\Autp$, where ${\cal A=A}(\amas,Q)$. Then $Q(f(\amas))\cong Q$. Let $M_i$ be an object in ${\rm ind}\, {\cal C}_Q$ such that $X_{M_i} = f(x_i) \in f(\amas)$ with $1\leq i\leq n.$ Due to the isomorphism $Q\cong Q(f(\amas)),$ the $M_i$'s form a local slice in the category ${\cal C}_Q$ and, in particular, $M_i$ lies in the transjective component of ${\cal C}_Q$.  The correspondence $P_i[1] \mapsto M_i$ extends to an automorphism $g$ of $\Gamma_{tr}$ (actually to an automorphism of ${\cal C}_Q$) and we get  $\phi(g)(x_i) = X_{g(P_i[1])} = X_{M_i} = f(x_i)$ for each $i.$ Therefore $\phi(g)=f.$
 \end{proof}

\begin{remark}\label{rem 3.11}
There is always a distinguished automorphism of the
transjective component, induced by the Auslander-Reiten translation $\tau $.
If $Q$ is a Dynkin quiver, then $\tau $ is periodic so this automorphism is
of finite order. If $Q$ is a representation-infinite quiver, then it is of
infinite order and $\mathbb{Z}$ is a subgroup of Aut$^{+}\mathcal{A}$.
\end{remark}

\begin{lemma}
\label{lemma_reflection}  Let $\mathcal{A}$ be a cluster algebra with a seed $(\amas,Q)$ such that $Q$ is a tree. Then
there exists an anti-automorphism $\sigma$ of  $\Gamma_{tr}$ given by a reflection with respect to a vertical axis. Moreover, we have $\sigma^2=1$ and $\sigma \tau^m \sigma = \tau^{-m}$ for any  $m \in {\mathbb Z}$, where $\tau$ stands for the Auslander-Reiten translation on $\Gamma_{tr}$. 
\end{lemma}

\begin{proof} We may assume without loss of generality that the quiver $Q$ of the initial seed is a tree.
 Let $M$ be an arbitrary point in a transjective component $\Gamma $
of $\Gamma \left( \mathcal{D}^{b}\left( \textup{mod}\,kQ\right) \right)$. We
recall that $\Gamma \cong \mathbb{Z}Q$ and we assume fixed a polarisation of 
$\Gamma $ (see \cite[p. 131]{ASS}). We define a reflection $\sigma _{\mathcal{D}}$
"along the vertical axis passing through $M$" in the following way.

There exists a unique slice $\Sigma ^{+}$ in $\Gamma $ having $M$ as its
unique source. This slice is the full subquiver of $\Gamma $ consisting of
all the points $N$ in $\Gamma $ such that there exists a path from $M$ to $N$
and every such path is sectional. Dually, one constructs the unique slice $%
\Sigma ^{-}$ in $\Gamma $ having $M$ as its unique sink.  Now there exists
an obvious bijection between the sets of points of $\Sigma ^{+}$ and $\Sigma
^{-}$, mapping each point of $\Sigma ^{+}$ to the unique point in $\Sigma
^{-}$ lying in the same $\tau$-orbit. Since $Q$ is a tree, the very definition of $\mathbb{Z}%
Q$ implies that this bijection first extends to an anti-isomorphism between $\Sigma ^{+}$
and $\Sigma ^{-}$ and then to an anti-automorphism $\sigma _{\mathcal{D}}$
of $\Gamma$.

Because $\sigma _{\mathcal{D}}$ is clearly compatible with the functor $\tau
^{-1}[1],$ it induces an anti-automorphism $\sigma $ of the transjective
component $\Gamma _{tr}$ of $\Gamma \left( \mathcal{C}_{Q}\right) .$
Finally, the asserted relations for $\sigma $ and $\tau $ follow easily. 
 \end{proof}

\begin{theorem}
\label{theorem_semidirect} {If $\mathcal{A}$ is a cluster algebra  with a seed $(\amas,Q)$ such that $Q$ is a tree then
the group of cluster automorphisms is the semidirect product 
${\rm Aut} \, {\cal A} = {\rm Aut}^+ {\cal A} \rtimes \ZZ_2$. This product is not direct.} 
\end{theorem}

\begin{proof}
{The anti-automorphism $\sigma $ given by Lemma \ref{lemma_reflection} induces a cluster automorphism $f_\sigma\in\Aut\setminus\aut$ of order two, and, therefore,  Corollary \ref{cor two new}  implies that ${\rm Aut} \, {\cal A} = {\rm Aut}^+ {\cal A} \rtimes \ZZ_2$. This product is not direct because $f_\sigma$  does not commute with $\tau$.}
 \end{proof}
%
%



\begin{corollary}{ Let $\mathcal{A}$ be a cluster algebra of Dynkin or euclidean type. Then $\Aut=\Autp\rtimes\ZZ_2$.}
\end{corollary} 
\begin{proof}{
The only case that is not a tree is the euclidean type $\tilde{\mathbb{A}}$, and, for this case, the result follows from Corollary \ref{cor anti new} and Example \ref{ex anti atilde}.}
\end{proof}

\subsection{Computing the automorphism groups for   quivers of   types ${\mathbb A}, {\mathbb D}, {\mathbb E}, \tilde{{\mathbb A}}, \tilde{{\mathbb D}}, \tilde{{\mathbb E}}$}
\label{sect 3.3}
Using Theorems \ref{theorem_main} and \ref{theorem_semidirect}, we are able to compute the cluster automorphism groups for the cluster algebras of Dynkin and euclidean types explicitly by computing the automorphism groups of the corresponding transjective component $\Gamma_{tr}$ of the cluster category. This computation is straightforward and is done by using the fact that under such an automorphism, the local structure of the quiver is preserved. We refer the reader to \cite{Ri} for a similar calculation. The results are collected in Table \ref{table}. As an example, we discuss the case of the quiver $\tilde{\mathbb D}_n$ in more detail.

\begin{table}[htdp]
\begin{center}
\begin{tabular}{c|c|c} Dynkin Type & ${\rm Aut}^+ {\cal A}$ & ${\rm Aut}\, {\cal A}$
\\\hline ${\mathbb A}_n, n>1$ & ${\mathbb Z}_{n+3}$ & $D_{n+3}$
\\\hline ${\mathbb A}_1$ & ${\mathbb Z}_{2}$ & $\mathbb{Z}_{2}$
\\\hline ${\mathbb D}_n,\;n>4$ & ${\mathbb Z}_n\times {\mathbb Z}_2$ & ${\mathbb Z}_n\times {\mathbb Z}_2 \rtimes {\mathbb Z}_2$ 
\\\hline ${\mathbb D}_4$ & ${\mathbb Z}_4 \times S_3$ & $D_4 \times S_3$
\\\hline ${\mathbb E}_6$ & ${\mathbb Z}_{14}$& $D_{14}$
\\\hline ${\mathbb E}_7 $ &${\mathbb Z}_{10}$& $D_{10}$
\\\hline$ {\mathbb E}_8$ & ${\mathbb Z}_{16}$& $D_{16}$
\\
\\Euclidean Type & ${\rm Aut}^+ {\cal A}$ & ${\rm Aut}\, {\cal A}$
\\\hline $\tilde{\mathbb A}_{p,q},\;p\neq q$  & $H_{p,q}$  & $H_{p,q} \rtimes {\mathbb Z}_2$
\\\hline $\tilde{\mathbb A}_{p,p}$  & $H_{p,p}\rtimes {\mathbb Z}_2$  & $H_{p,p} \rtimes {\mathbb Z}_2 \rtimes {\mathbb Z}_2$ 
\\\hline $\tilde{\mathbb D}_{n-1}, \; n\neq5$ & $G$ & $G\rtimes {\mathbb Z}_2$ 
\\\hline   $\tilde{\mathbb D}_4$ & ${\mathbb Z} \times S_4$  & ${\mathbb Z} \times S_4\rtimes {\mathbb Z}_2$ 
\\\hline  $ \tilde{\mathbb E}_6$ & ${\mathbb Z} \times S_3$ & ${\mathbb Z} \times S_3 \rtimes {\mathbb Z}_2$ 
\\\hline   $\tilde{\mathbb E}_7$ & ${\mathbb Z} \times {\mathbb Z}_2$ & $({\mathbb Z}\times {\mathbb Z}_2) \rtimes {\mathbb Z}_2$
\\\hline   $\tilde{\mathbb E}_8$  & ${\mathbb Z}$ & ${\mathbb Z} \rtimes {\mathbb Z}_2$ 
\end{tabular} \caption{Cluster automorphism groups for quivers of Dynkin and euclidean types}
\end{center}
\label{table}
\end{table}

The notation in Table \ref{table} is as follows: $D$ stands for the dihedral group, $H_{p,q} = \langle  r_1, r_2| r_1r_2 = r_2 r_1, r_1^p=r_2^q \rangle$ and $G$ is given by equation  (\ref{G}) below. 

Note that for the type $\mathbb{A}_1$, the quiver $Q$ has no arrows, and therefore there are no anti-automorphisms in this case. Moreover, the cluster algebra has exactly two clusters, each consisting of a single cluster variable and the Auslander-Reiten translation in the cluster category is of order 2.

In the case of $\tilde{\mathbb{A}}_{p,q}$, the automorphisms $r_{1}$ and $
r_{2}$ are defined as follows. Let $\Sigma ^{+}$, as in the proof of Lemma \ref{lemma_reflection}, be the
unique slice in the transjective component having $M$ as its only source.
Let $\omega $ denote the unique path of length $p$ from $M$ to the only
sink in $\Sigma^+$. The automorphism $r_{1}$ is defined as follows. We send $M$ to the
target $N$ of the only arrow on $\omega $ with source $M$, and send $\Sigma
^{+}$ isomorphically to the unique slice having $N$ as its only source. This
map extends to an automorphism $r_{1}$ of the transjective component. We
define $r_{2}$ similarly, using the path of length $q$ from $M$ to the sink
of $\Sigma ^{+}$. Note that $\tau^{-1} =r_{1}r_{2}.$

\begin{example}
Consider a cluster algebra ${\cal A=A}(\amas, Q)$, where $Q$ is the following euclidean quiver of   type $\widetilde{\mathbb D}_{n-1}$ with $n\neq 5$ (the number of points is $n$).

\begin{equation*}
\xymatrix@R10pt{
1\ar[rd]&&&&&&n-1\\
&3\ar[r]&4\ar[r]&\quad\cdots\quad\ar[r]&n-3\ar[r]&n-2\ar[ru]\ar[rd]\\
2\ar[ru]&&&&&&n
}
\end{equation*}

The transjective component $\Gamma_{tr}$ is represented by the infinite translation quiver  $\mathbb{Z}Q$,  an example with $n=8$ is given in Figure \ref{fig 22}.
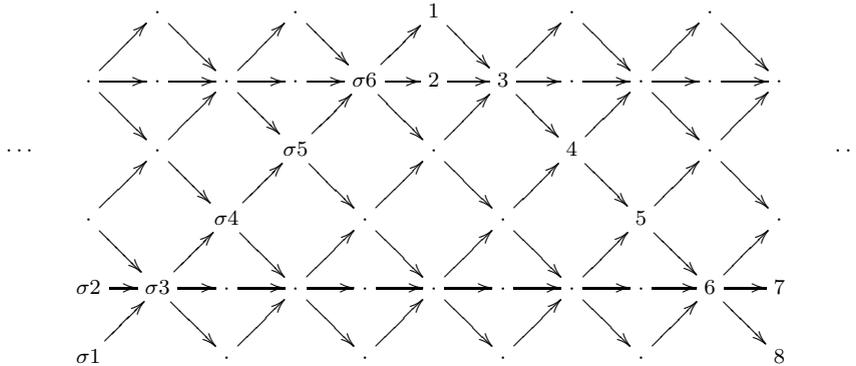
\begin{figure}
\[\xymatrix@R10pt@C10pt@!{&&\cdot\ar[rd]&&\cdot\ar[rd]&&1\ar[rd] &&\cdot\ar[rd] &&\cdot\ar[rd]\\
&\cdot \ar[ru]\ar[r]\ar[rd]&\cdot\ar[r]&\cdot \ar[ru]\ar[r]\ar[rd] & \cdot \ar[r] &\zs 6   \ar[ru]\ar[r]\ar[rd] & 2
 \ar[r] &3 \ar[ru]\ar[r]\ar[rd] & \cdot \ar[r]  &\cdot \ar[ru]\ar[r]\ar[rd] & \cdot \ar[r] &\cdot\\
\cdots&&\cdot \ar[ru]\ar[rd] && \zs 5\ar[ru]\ar[rd] &&\cdot \ar[ru]\ar[rd] &&4\ar[ru]\ar[rd] &&\cdot \ar[ru]\ar[rd]&& \cdots \\
&\cdot\ar[ru]\ar[rd] &&\zs 4\ar[ru]\ar[rd] &&\cdot\ar[ru]\ar[rd] &&\cdot\ar[ru]\ar[rd] &&5\ar[ru]\ar[rd] &&\cdot\\
&\zs 2\ar[r] &\zs 3 \ar[ru]\ar[rd]\ar[r] & \cdot\ar[r] &  \cdot \ar[ru]\ar[rd]\ar[r] & \cdot\ar[r] &  \cdot \ar[ru]\ar[rd]\ar[r] & \cdot\ar[r] &  \cdot \ar[ru]\ar[rd]\ar[r] & \cdot\ar[r] & 6 \ar[ru]\ar[rd]\ar[r] & 7  &\\
&\zs 1\ar[ru] &&\cdot\ar[ru] &&\cdot\ar[ru] &&\cdot\ar[ru] &&\cdot\ar[ru] &&8 
}
\]
\caption{ Auslander-Reiten quiver $\zG_{tr}$ of the transjective component of type $\tilde{ \mathbb{D}}$}\label{fig 22}
\end{figure}
%
Recall that automorphisms of a translation quiver are by definition the quiver automorphisms commuting with the translation $\tau$.  Besides $\tau$ we   introduce three more generators of the automorphism group.

Let $\rho_1$ denote the automorphism which interchanges the corresponding points in the $\tau$-orbits of $1$ and $2$ in $\zG_{tr}$ and fixes all other points. 
Let $\rho_n$ denote the automorphism which interchanges the corresponding points in the $\tau$-orbits of $n-1$ and $n$ in $\zG_{tr}$ and fixes all other points. 
Finally, let $\zs$ be the automorphism given by the translation of the plane  that sends the point $n$ to the point $1$ followed by the reflection with respect to the horizontal line through the point $1$; we have indicated the action of $\zs$ on the points $1,2,\ldots,6$ in Figure \ref{fig 22}, and $\zs 8=1$ and $\zs 7=2$.

%


 Since for every point of the quiver $kQ$ the number of incoming and outgoing arrows is preserved under a quiver automorphism, one sees that every automorphism of $kQ$  can be expressed as a combination of $\tau,\rho_1,\rho_n$, and $\zs.$

We note the following relations between these generators: 
\begin{enumerate}
\item the translation $\tau$ commutes with all automorphisms and is of infinite order; 
\item $\rho_1$ and $\rho_n$ are of order two and commute with each other;
\item $\zs^2=\tau^{n-3}$;
\item $\rho_1\zs=\zs\rho_n$ and $\zs\rho_1=\rho_n\zs$.
\end{enumerate}

Thus we get a presentation of the group of automorphisms of $\zG_{tr}$ as 
\begin{equation}\label{G}
G=\left\langle \tau,\zs,\rho_1,\rho_n \left| \begin{array}{c}
 \rho_i^2=1 ,\tau \rho_i=\rho_i \tau \ (i=1,n)\\
\tau\zs=\zs\tau,\ \zs^2=\tau^{n-3} \\
\rho_1\zs=\zs\rho_n, \ \zs\rho_1=\rho_n\zs
\end{array}\right.\right\rangle
\end{equation}

\end{example}

\begin{section}
{Cluster algebras from surfaces}\label{sect 3} Following \cite{FST}, we describe the   construction of cluster algebras from surfaces. 

Let $S$ be an oriented Riemann surface with or without boundary, and let $M\subset S$ be a finite set of marked points such that $M$ contains at least one point of every connected component of the boundary. If the boundary is empty then the surface is said to be \emph{closed}. Points in $M$ that are in the interior of $S$ are called \emph{punctures}. We call the pair $(S,M)$ simply a \emph{surface}.

For technical reasons, we require that $(S,M)$ is not
a sphere with one, two or three punctures;
a disc with one, two or three marked points on the boundary;
or a punctured disc with one marked point on the boundary. Some simple examples of surfaces are given in Table \ref{table 1}.

  
\begin{table}
\begin{center}
  \begin{tabular}{ c | c | c | c  | l | l  }
  \  g\ \  &\ \  b \ \   & \ \  c \ \ &\ \ p\ \     &\  surface  & type \\
   \hline    
  0& 1 & n+3 &0& \ polygon & $\mathbb{A}_n$\\
    0 & 1 & n &1& \ once punctured polygon & $\mathbb{D}_n$ \\ 
    0 & 1 & n-3 &2&\ twice punctured polygon& $\tilde{\mathbb{D}}_{n-1}$ \\ 
    0& 2 & n & 0&\ annulus& $\tilde{\mathbb{A}}_{n-1}$\\
    0 & 2 & n-3 & 1&\ punctured annulus& not acyclic\\
        0 & 0 & 0 & 4 & \ sphere with 4 punctures& not acyclic\\
1 & 0 & 0 & 1 &\ punctured torus& not acyclic\\
        3 & 0 & n-3 & 0& \ pair of pants & not acyclic\\ \\
  \end{tabular}
\end{center}
\caption{Examples of  surfaces, $g$ is the genus, $b$ is the number of boundary components, $p$ the number of punctures, $c$ the number of marked points on the boundary and $n$ is the rank of the cluster algebra.}\label{table 1}
\end{table}

\subsection{Arcs and triangulations}

\begin{definition}
An \emph{arc} $\zg$ in $(S,M)$ is the isotopy class of a curve in $S$ such that 
\begin{itemize}
\item[(a)] the endpoints of the curve are in $M$;
\item[(b)]  the curve does not cross itself, except that its endpoints may coincide;
\item[(c)] except for the endpoints,  the curve is disjoint from $M$ and
  from the boundary of $S$,
\item[(d)]  the curve does not cut out an unpunctured monogon or an unpunctured bigon. 
\end{itemize}   
\end{definition}

For any two arcs $\zg,\zg'$ in $S$, let $e(\zg,\zg')$ be the minimal
number of crossings of 
curves $\za$ and $\za'$, where $\za$ 
and $\za'$ range over all curves in the isotopy classes 
$\zg$ and $\zg'$, respectively.
We say that two arcs $\zg$ and $\zg'$ are  \emph{compatible} if $e(\zg,\zg')=0$.

An \emph{ideal triangulation} is a maximal collection of
pairwise compatible arcs.
The arcs of an ideal
triangulation cut the surface into \emph{ideal triangles}.

\begin{figure}
\begin{center}
\includegraphics{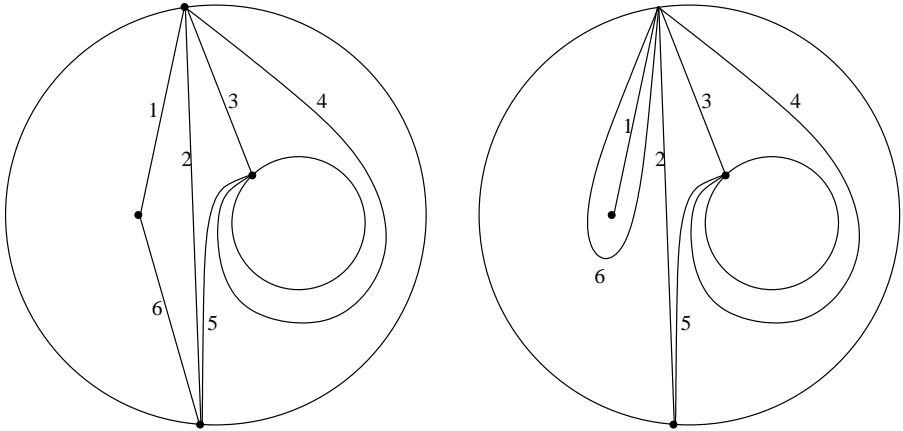}
\caption{Two ideal triangulations of a punctured annulus related by a flip of the arc $6$. The triangulation on the right hand side has a self-folded triangle.}
\label{fig 2}
\end{center}
\end{figure}

 Examples of ideal triangulations are given in Figure \ref{fig 2}.
There are two types of ideal triangles: triangles that have three distinct sides and triangles that have only two. The latter are called \emph{self-folded} triangles. For example, the triangle formed by the arcs $6$ and $1$ on the right hand side of Figure \ref{fig 2} is self-folded.
In a self-folded triangle the arc incident to the puncture is called \emph{radius} and the other arc is called \emph{loop}.

Ideal triangulations are connected to each other by sequences of 
{\it flips}.  Each flip replaces a single arc $\gamma$ 
in a triangulation $T$ by the (unique) arc $\gamma' \neq \gamma$
that, together with the remaining arcs in $T$, forms a new ideal
triangulation.

Note that an arc $\gamma$ that is the radius inside a self-folded triangle
in $T$ cannot be flipped.

In \cite{FST}, the authors associated a cluster algebra to 
any bordered surface with marked points.  Roughly speaking,
the cluster variables correspond to arcs, the clusters
to triangulations, and the mutations to flips.  However,
because arcs inside self-folded triangles cannot be flipped,
the authors were led to introduce the slightly more general notion 
of {\it tagged arcs}.  They showed that ordinary arcs can
all be represented by tagged arcs and gave a notion of flip 
that applies to all tagged arcs.

A {\it tagged arc} is obtained by taking an arc that does not 
cut out a once-punctured monogon and marking (``tagging")
each of its ends in one of two ways, {\it plain} or {\it notched},
so that
\begin{itemize}
\item[(a)] each end connecting to a marked point on the boundary of $S$ must be tagged plain;
\item[(b)] if  both ends of an arc connect to the same point then they must be tagged in the same way.
\end{itemize}

\begin{definition} \label{compatible}
Two tagged arcs $\alpha$ and
$\beta$ are called {\it compatible} if the arcs $\alpha^0$ and $\beta^0$ obtained from 
   $\alpha$ and $\beta$ by forgetting the taggings are compatible and
\begin{itemize}
\item[(a)] if $\alpha^0=\beta^0$ then at least one end of $\alpha$
  must be tagged in the same way as the corresponding end of $\beta$;
\item[(b)] if $\alpha^0\neq \beta^0$ but they share an endpoint $a$, 
 then the ends of $\alpha$ and $\beta$ connecting to $a$ must be tagged in the 
same way.
\end{itemize}
\end{definition}

One can represent an ordinary arc $\beta$ by 
a tagged arc $\iota(\beta)$ as follows.  If $\beta$ 
does not cut out a once-punctured monogon, then $\iota(\beta)$
is simply $\beta$ with both ends tagged plain.
Otherwise, $\beta$ is a loop based at some marked point $a$
and cutting out
a punctured monogon with the sole puncture $b$ inside it.
Let $\alpha$ be the unique arc connecting $a$ and $b$ and compatible
with $\beta$.  Then $\iota(\beta)$
is obtained by tagging $\alpha$ plain at $a$ and notched at $b$.
Figure \ref{figtags} shows the tagged triangulation corresponding to the triangulation on the right hand side of Figure \ref{fig 2}. The notching is 
indicated by a 
bow tie.

  A maximal  collection
of pairwise compatible tagged arcs is called a {\it tagged triangulation}.

\begin{figure}
\begin{center}
\includegraphics{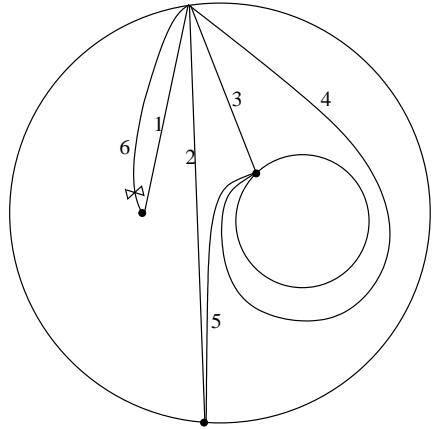}
\caption{Tagged triangulation of the punctured annulus corresponding to the ideal triangulation of the right hand side of Figure \ref{fig 2}.}
\label{figtags}
\end{center}
\end{figure}

We are now  ready to define the cluster algebra associated to the surface. For that purpose, we choose an ideal triangulation $T$ and then define a quiver $Q_T$ without loops or 2-cycles, or, equivalently, a skew-symmetric integer matrix $B_T$.
Let $\tau_1,\tau_2,\ldots,\tau_n$ be the $n$ arcs of
$T$.
For any triangle $\Delta$ in $T$ which is not self-folded, we define a matrix 
$B^\Delta=(b^\Delta_{ij})_{1\le i\le n, 1\le j\le n}$  as follows.
\begin{itemize}
\item $b_{ij}^\Delta=1$ and $b_{ji}^{\Delta}=-1$ in each of the following cases:
\begin{itemize}
\item[(a)] $\tau_i$ and $\tau_j$ are sides of 
  $\Delta$ with  $\tau_j$ following $\tau_i$  in the 
  clockwise order;
\item[(b)] $\tau_j$ is a radius in a self-folded triangle enclosed by a loop $\tau_\ell$, and $\tau_i$ and $\tau_\ell$ are sides of 
  $\Delta$ with  $\tau_\ell$ following $\tau_i$  in the 
clockwise order;
\item[(c)] $\tau_i$ is a radius in a self-folded triangle enclosed by a loop $\tau_\ell$, and $\tau_\ell$ and $\tau_j$ are sides of 
  $\Delta$ with  $\tau_j$ following $\tau_\ell$  in the 
clockwise order;
\end{itemize}
\item $b_{ij}^\Delta=0$ otherwise.
\end{itemize}
 
Then define the matrix 
$ B_{T}=(b_{ij})_{1\le i\le n, 1\le j\le n}$  by
$b_{ij}=\sum_\Delta b_{ij}^\Delta$, where the sum is taken over all
triangles in $T$ that are not self-folded.

Note that  $B_{T}$ is skew-symmetric and each entry  $b_{ij}$ is either
$0,\pm 1$, or $\pm 2$, since every arc $\tau$ is in at most two triangles. 
We associate a quiver $Q_T$ to the matrix $B_T$ as follows. The points of $Q_T$ are labeled by $1,2,\ldots,n$ and the number of arrows from $i$ to $j$ equals $b_{ij}$, with the convention that if $b_{ij}$ is a negative number, then having $b_{ij}$ arrows from $i$ to $j$ means having $|b_{ij}|$ arrows from $j$ to $i$. For example, the quiver corresponding to the triangulation on the right of Figure \ref{fig 2} is 
\[\xymatrix@C50pt@R10pt{1\ar[rd]&&3\ar[rd]\ar[dd]\\&2\ar[ru]&&4\\
6\ar[ru]&&5\ar[ul]\ar[ru]}\]
 Since the matrix $B_T$ is skew-symmetric, it follows that $Q_T$ has no oriented cycles of length at most two.

The cluster algebra $\mathcal{A}=\mathcal{A}(\mathbf{x},Q_T)$ given by the quiver $Q_T$ is said to be the \emph{cluster algebra} (with trivial coefficients) \emph{associated to the surface $(S,M)$}.

\begin{remark}\label{remtag} \ralf{It has been shown in \cite{FST} that} if the surface $(S,M)$ is not a closed surface with exactly one puncture then there is a bijection between tagged arcs in $(S,M)$ and cluster variables in the cluster algebra, such that compatible tagged arcs correspond to compatible cluster variables, and tagged triangulations correspond to clusters. The mutations in the cluster algebra are given by the flips in the tagged triangulations.

If $(S,M)$ is a closed surface with exactly one puncture, for example a once punctured torus, then there is a bijection between arcs (not tagged arcs) and cluster variables. The reason for this is that a flip can not change the tagging at the endpoint of a given arc because all arcs are incident to the unique puncture, thus changing the tagging on one of the arcs would not be compatible with the others. 

This fact will be important  when we consider the cluster automorphisms induced by change of taggings, see Lemma \ref{lemtag} and Theorem \ref{theorem mg}.
\end{remark}

\subsection{Mapping class groups} \label{sect 4.2} In this section, we give the definitions and some basic properties of mapping class groups. For further details we refer the reader to \cite{FM}.

Let $\textup{Homeo}^+(S)$ be the group of orientation preserving homeomorphisms from $S$ to $S$ and 
let $\textup{Homeo}^+(S,\partial S)$ be the subgroup of all   $f\in  \textup{Homeo}^+(S)$ such that the restriction  $f|_{\partial S}$ of $f$ to the boundary is equal to the identity $1_{\partial S}$.

Two homeomorphisms $f,g$ of $S$ are \emph{isotopic} if there is a continuous function $H:S\times [0,1]\to S$ such that $H(x,0)=f$ and $H(x,1)=g$ and such that for each $t\in [0,1]$ the map $H(x,t):S\to S$ is a  homeomorphism.

Let $\textup{Homeo}_0(S,\partial S)$ be the subgroup of all $f\in \textup{Homeo}^+(S,\partial S)$ that are isotopic to the identity $1_S$ via an isotopy $H$ fixing $\partial S$ pointwise, thus $H(x,t)=x$ for all $x\in \partial S$ and $t\in [0,1]$.

The \emph{mapping class group} $\Mod$ of the surface $S$ is defined as the quotient group
 \[\Mod=\textup{Homeo}^+(S,\partial S)/\textup{Homeo}_0(S,\partial S).\]

 We   now define the mapping class group of the surface with marked points $(S,M)$ in a similar way. 
 Let $\homeo$ be the group of orientation preserving homeomorphisms from $S$ to $S$ which map $M$ to $M$. Note that we do \emph{not} require that the points in $M$ are fixed, neither that the points  on the boundary of $S$ are fixed, nor that each boundary component is mapped to itself. However if a boundary component $C_1$ is mapped to another component $C_2$ then the two components must have the same number of marked points.
We say that a homeomorphism $f$ is \emph{isotopic to the identity relative to $M$}, if $f$ is isotopic to the identity via an  isotopy $H$ that fixes $M$ pointwise, thus $H(x,t)=x$ for all $x\in M$ and $t\in [0,1]$. Let $\homeozero$ be the subgroup of all $f\in \homeo$ that are isotopic to the identity relative to $M$.

We define the \emph{mapping class group} $\mg$ of the surface $(S,M)$ to be the quotient
\[\mg=\homeo/\homeozero.\]
 
 The two mapping class groups are related as follows.
\begin{lemma}\label{mgg}
$\Mod $ is isomorphic to a subgroup of $\mg$.
\end{lemma}
\begin{proof}
Clearly $\textup{Homeo}_0(S,\partial S)$ is a subgroup of $\homeozero$, and thus there is a map from $\Mod$ to $\mg$ sending the class of a homeomorphism $f$ in $\Mod$ to its class in $\mg$. This map is an injective group homomorphism.
 \end{proof}

 Next we review \emph{Dehn twists}. 
 If $S$ is an annulus, we can parametrise $S$ as $S^1\times [0,1]$, where $S^1$ denotes the circle of radius one, such that the two boundary components of $S$ are $S^1\times\{0\}$ and $S^1\times\{1\}$. Then the map $T:S^1\times [0,1]\to S^1\times[0,1], (\theta,t)\mapsto (
\theta +2\pi t, t)$ is an orientation preserving homeomorphism that fixes both boundary components of $S$ pointwise. $T$ is called \emph{Dehn twist} on $S$. Thus $T\in\homeos$, and since $T$ is not isotopic to the identity relative to $\partial S$, the class of $T$ in  $\Mod$ is non-zero. It is clear that the class of $T$ has infinite order in $\Mod$, and hence it generates an infinite cyclic subgroup of $\Mod $. Since $S$ is an annulus, one can show that the class of $T$ actually generates the whole group $\Mod $. One can think of this Dehn twist as cutting the annulus along the equator $S^1\times\{1/2\}$, performing a full rotation of one end (keeping the boundary fixed) and gluing the two pieces back together, see Figure \ref{fig dehn}.

\begin{figure} 
\includegraphics{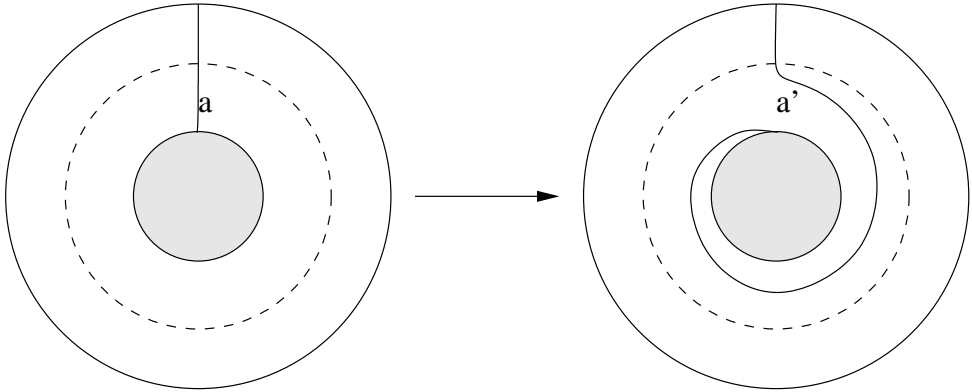}
\caption{Dehn twist on the annulus, the curve a is mapped to the curve a'; the equator is drawn as a dashed line} \label{fig dehn}
\end{figure}
 
Now suppose that $S$ is any Riemann surface, and that $c$ is a closed simple curve in $S$. Then one can define a Dehn twist about $c$ in $S$ by performing the Dehn twist $T$ on a regular neighbourhood $N$ of $c$ in $S$ which is homeomorphic to an annulus, see Figure \ref{fig regular}.

\begin{figure} 
\begin{center}
\includegraphics{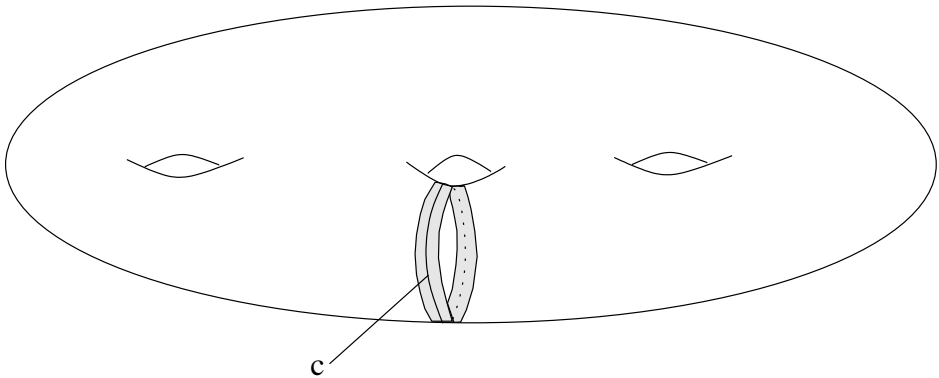}
\end{center}\caption{Regular neighborhood of a closed curve $c$} \label{fig regular}
\end{figure}

\begin{remark}\label{rem mg1}
If the surface $S$ has genus at least one or if $S$ has genus zero and at least two boundary components, then there exists a Dehn twist that generates an infinite cyclic subgroup of $\Mod$.
\end{remark} 
 
\begin{remark}\label{rem mg2} We list the mapping class groups of some Riemann surfaces: 
\begin{center}
\begin{tabular}{|c|c|c|c|c|c|}\hline
S& disc & annulus& punctured disc & torus\\\hline
$\mathcal{M}od (S)$ &
 $0$ &$\ZZ$&$0$&$\textup{SL}(2,\ZZ)$ \\ \hline
\end{tabular}
\end{center}

\begin{center}
\begin{tabular}{|c|c|c|c|c|c|} \hline
S& $\textup{sphere with $3$ punctures}$& $\textup{disc with $p$ punctures}$\\ \hline
$\mathcal{M}od (S)$ &
 $S_3$ &$B_p$ \\ \hline
\end{tabular}
\vspace{10pt}
\end{center}
where $S_3$ denotes the symmetric group on $3$ letters and $B_p$ the braid group on $p$ strands.

In general mapping class groups are very
difficult to compute and known only for a few cases.

%

\end{remark} 
 
\subsection{Marked mapping class group}
In order to describe the cluster automorphism group 
of a cluster algebra corresponding to a marked surface $\left( S,M\right) $, we need a group that contains the mapping class group $\mg$ from the previous subsection, but also contains automorphisms that change the taggings at the punctures. We call this group the \emph{marked mapping class group.}


\begin{definition}
A \emph{marked mapping class} $(\bar f, \cal P) $ is an element $\bar f\in \mg$ together with a subset $\cal P$ of the set of punctures of $(S,M)$. 
\end{definition}
If the set $\cal P$ consists of a single element $z$, then we   write $(\bar f, z)$ instead of $(\bar f, \{z\})$ for the marked  mapping class.

A marked  mapping class acts on the  set of arcs of the surface by applying the homeomorphism $f$ and changing the tagging at each puncture in the set $\mathcal P$. We can define a product on the set of marked  mapping classes by
\[(\bar f_1,\mathcal{P}_1)(\bar f_2, \mathcal {P}_2) =(\bar f_1\bar f_2, \mathcal {P}_1\ominus f_1(\mathcal {P}_2)),\]
where $\ominus $ denotes the symmetric difference $A\ominus B= (A\cup B)\setminus (A\cap B)$.

\begin{lemma}
The set of marked  mapping classes forms a group under the product above. 
\end{lemma}
\begin{proof}
Associativity follows from the associativity of the symmetric difference, the identity is given by $(1, \emptyset)$, where $1$ denotes the identity of $\mg$, and inverses are given by $(\bar f,\mathcal{P})^{-1}=(\bar f^{-1},f^{-1}(\mathcal{P}))$.
 \end{proof}

\begin{definition}
The \emph{marked mapping class group} $\mgm$ of the surface $(S,M)$ is the group of all marked  mapping classes of $(S,M)$.
\end{definition}

We can also define $\mgm$ as a semidirect product as follows. Let $\{z_1,\ldots,z_p\}$ be the set of punctures and let $\mathcal{Z}$ be the group of powersets of $\{z_1,\ldots,z_p\}$ with respect to the group operation $\ominus$.
Note that $\mathcal{Z}\cong\ZZ_2^p$.
For each $\fbar \in \mg$, the homeomorphism $f$ induces an automorphism of $\mathcal{Z}$. This defines an action of $\mg$ on $\mathcal{Z}$.
\begin{lemma}\label{lem sd}
$\mgm$ is isomorphic to the semidirect product $\mathcal{Z}\rtimes \mg$ with respect to the above action.
\end{lemma}
\begin{proof}
The product in $\mathcal{Z}\rtimes\mg$ is  defined as
\[(\fbar_1,\mathcal{P}_1)(\fbar_2,\mathcal{P}_2)=(\fbar_1\fbar_2,\mathcal{P}_1\ominus f_1(\mathcal{P}_2)),\]
which proves the statement.
 \end{proof}

\begin{remark}\label{rem 38}
If the surface has precisely one puncture, then for each $\fbar\in \mg$ the homeomorphism $f$ must fix the puncture, hence the action of $\mg$ on  $\mathcal{Z}$ is trivial, whence $\mgm$ is the direct product $\mathcal{Z}\times \mg$.
\end{remark}

\begin{corollary}\label{cor 37}
\begin{itemize}
\item[{\rm (1)}] $\mg$ is a subgroup of $\mgm$;
\item[{\rm (2)}] $\mathcal{Z}$ is a normal subgroup of $\mgm$;
\item[{\rm (3)}] $\mgm$ is generated by the elements $(\bar f, z)$ where $\bar f$ runs over all elements of $\mg$ and $z$ runs over all punctures;
\item [{\rm (4)}] $(1,z)(1,z)=(1,\emptyset)$, more generally $(1,\mathcal{ P})^m=(1, \emptyset)$ if $m$ is even, and $(1,\mathcal{ P})^m=(1,\mathcal{P})$ if $m$ is odd;
\item[{\rm (5)}] If $z,z'$ are two punctures such that  $f$ maps  $z$ to $z'$, then
\[ \big(1, z'\big)\,\big(\bar f,\emptyset\big) = \big(\bar f,\emptyset\big)\,\big(1,z\big).\]
\end{itemize}
\end{corollary}
\begin{proof}
(1),(2) and (3) are direct consequences of Lemma \ref{lem sd}, and (4) and (5)  are easy computations.
 
\end{proof}

 \subsection{Mapping class group and cluster automorphism group}
 
We now show that  the group of cluster automorphisms has a subgroup isomorphic to  $\mgm$. The change of tagging induces a cluster automorphism which is described in the following Lemma.

Let $z$ be a puncture in $(S,M)$ and $\za$ any arc. We denote by $\za^z$ the arc that is isotopic to $\za$ and has the opposite tagging at each endpoint that is equal to $z$. Essentially there are three different cases which are illustrated in Figure \ref{figtag}, namely $\za$ can have one endpoint, both endpoints or no endpoint equal to $z$.

\begin{figure}
\begin{center}
\scalebox{0.8}{\input{figtag.pstex_t}}
\caption{Change of tagging at $z$}\label{figtag}
\end{center}
\end{figure}
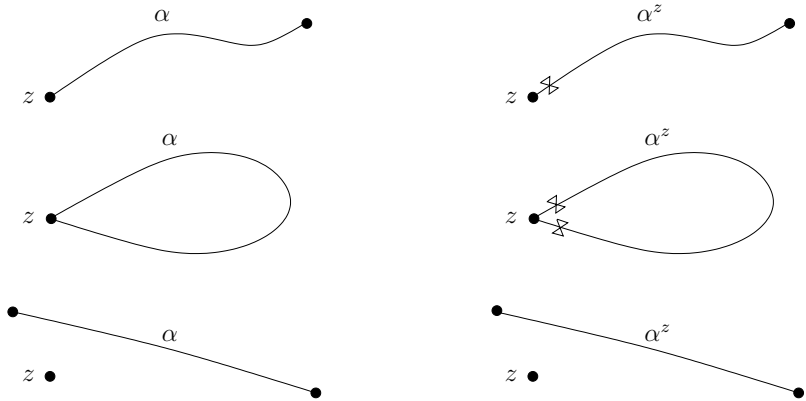

 \begin{lemma}
\label{lemtag} Assume $(S,M)$ is not a closed surface with exactly one puncture, and assume $T$ is a triangulation of $(S,M)$. Then,
for every puncture $z$ of $(S,M)$, the automorphism $\psi_z:\mathcal{A}\to\mathcal{A}$ defined by $\psi_z(x_\tau)=x_{\tau^z}$, for every arc $\tau\in T$, where $x_\tau$ is the cluster variable corresponding to $\tau$, and
extended to the other cluster variables by the algebra homomorphism
properties, is a cluster automorphism in $\aut$.
\end{lemma}
\begin{proof}
The cluster algebra $\mathcal{A}$ has an initial seed $(\mathbf{x}_T,B_T)$ associated to the triangulation $T$. 
As noted in \cite[Definition 9.2]{FST}, the compatibility of tagged arcs is 
invariant with respect to a simultaneous change of all tags at a given
puncture, therefore the set $T'=\{\tau_1^z,\ldots,\tau_n^z\}$ also is a triangulation, and hence it  defines  another seed $(\mathbf{x}_{T'},B_{T'})$ of $\mathcal{A}$, where 
\[ B_{T'}=\left(b_{\tau_i^z\tau_j^z}\right)_{ij}=\left(b_{\tau_i\tau_j}\right)_{ij} =B_T.\]
The automorphism $\psi_z$ sends the seed $(\mathbf{x}_T,B_T)$ to the seed $(\mathbf{x}_{T'},B_{T'})$, which shows (CA1). 
Moreover, since $B_{T^{^{\prime }}}=B_{T},$ the quivers
corresponding to these two matrices are equal. By Lemma \ref{lemma1}, $\psi _{z}\in $
Aut$^{+}\mathcal{A}$.
%
 \end{proof} 
 
\begin{remark}
In the case of a closed surface with exactly one puncture, $\psi_z$ is not defined. Indeed, in this case all arcs of a triangulation start and end at the puncture and thus must be tagged in the same way. Therefore flips do not change the tagging.  
\end{remark}

\begin{theorem}
\label{theorem mg} 
Let $(S,M)$ be a surface with $p$ punctures. 
Then $\mg$ is isomorphic to a subgroup of $\aut$. If moreover $p\ge 2$, or if $\partial S\ne\emptyset$, then $\mgm$ is isomorphic to a subgroup of $\aut$.
\end{theorem}
\begin{proof}
We start by showing that $\mg$ is isomorphic to a subgroup of $\aut$. Fix a triangulation $T$ and let $\mathbf{x}_T$ be the  corresponding cluster. Denote  the elements of $T$ by $\tau_1,\ldots,\tau_n$ and those of $\mathbf{x}_T$ by $x_{\tau_1},\ldots,x_{\tau_n}$.   Then $\{x_{\tau_1},\ldots,x_{\tau_n}\}$ is a transcendence basis of the ambient field $\mathcal{F}$ of the cluster algebra. 

For any element $f\in \homeo$, let $\fbar\in\mg$ denote its   class in the mapping class group. 
Define a map $\phi:\mg\to \aut$ by letting $\phi(\fbar)$ be the map from $\mathcal{A}$ to $\mathcal{A}$ defined on the basis by $\phi(\fbar)(x_{\tau_i})=x_{f(\tau_i)}$ and extended to $\mathcal{A}$ by the algebra homomorphism properties. 
We   show that $\phi$ is an injective group homomorphism.

To show that the definition of $\phi $ does not depend on the choice of the representative $f$, assume that $f$ is isotopic to the identity relative to $M$. Then $f(\tau_i)$ is isotopic to $\tau_i$ relative to $M$, and thus, $f(\tau_i)$ and $\tau_i$ represent the same arc. It follows that $\phi$ does not depend on the choice of $f$.

Next we show that $\phi(\fbar)$ is a cluster automorphism.
Since $f$ is a homeomorphism, any two arcs $\za,\zb$ in $(S,M)$ which are compatible, have compatible images $f(\za),f(\zb)$ in $(S,M)$. Therefore any triangulation of $(S,M)$ is mapped under $f$ to a triangulation of $(S,M)$. 
 Thus $\phi \left( \overline{f}\right) $ maps clusters to clusters.
By Corollary \ref{cor2.7bis},  $\phi \left( \overline{f}\right) $ is a cluster
automorphism. 
%
Moreover, 
since $f$ is actually an orientation preserving homeomorphism, we have $\phi(\fbar)\in \aut$, and this shows that $\phi$ is well-defined.

Now we show that the definition of $\phi$ does not depend on the choice of the triangulation $T$. That is, let us show that for any 
arc 
 $\za $ in $(S,M)$
\begin{equation}
\label{stern}
x_{f(\za)}=\phi(\fbar)(x_\za).
\end{equation}
Indeed, let $\za $ be an arc in $(S,M)$. Then there is a sequence of flips $\mu_T$ such that $\mu_T\,\tau=\za$ for some $\tau\in T$. Let $\mubar$ denote the corresponding sequence
of mutations in $\mathcal{A}$. Then
$x_{f(\za)}= x_{f(\mu_T\,\tau)}=x_{\mu_{f(T)}f(\tau)}$, where the last identity holds because $f$ commutes with flips. Since flips correspond to mutations, we get
$x_{f(\za)}= \mubar_{\phi(\fbar)(\mathbf{x}_T)}x_{f(\tau)}$, 
which by definition of $\phi(\fbar)$ is equal to 
$\mubar_{\phi(\fbar)(\mathbf{x}_T)}\phi(\fbar)(x_{\tau})$.
Now using the fact that $\phi(\fbar)$ is a cluster automorphism, we get 
$x_{f(\za)}= \phi(\fbar)(\mubar_{\mathbf{x}_T}(x_\tau))$, which is equal to 
$\phi(\fbar)(x_{\mu_{T}(\tau)})=\phi(\fbar)(x_\za)$, because flips correspond to mutations.

 To show that $\phi$ is a group  homomorphism, let $\fbar,\gbar\in \mg$, then for any arc $\za$,
 \[\phi(\fbar\circ\gbar)(x_\za)=x_{f\circ g (\za)} =\phi(\fbar)(x_{g(\za)})=\phi(\fbar)\big(\phi(\gbar)(x_\za)\big).
 \]

Finally, we show that $\phi $ is injective. Let $\fbar\in\textup{Ker}\,\phi$. Then $\phi(\fbar)=1_{\mathcal{A}}$, and for any arc $\za$ in $(S,M)$, we have $x_\za=\phi(\fbar)(x_\za)=x_{f(\za)}$, where the last identity holds by equation (\ref{stern}). Thus $f(\za)$ is isotopic to $\za$, for every arc $
\za$ in $(S,M)$, and in particular, $f $ fixes each point in $M$. Thus for every triangle $\zD$, $f$ fixes the points of $\zD$  and maps the arcs of $\zD$ to isotopic arcs. Therefore $f$ is isotopic to  the identity on each triangle, and hence on the whole surface. This shows
that $f$ is zero in $\mg$, and hence $\phi$ is injective.
This shows that $\mg$ is isomorphic to a subgroup of $\aut$. 

Now suppose that $p\ge2$ or $\partial S \ne \emptyset$, For any puncture $z$   let $\psi_z$ be the cluster automorphism of Lemma \ref{lemtag}. 
Define a map $\chi:\mgm\to \aut$ by $\chi(\fbar,\mathcal{P})=(\prod_{z\in\mathcal{P}}\psi_z)\phi(\fbar)$.
In order to show that $\chi$ is a group homomorphism, we compute
\begin{equation}\label{eq44} \chi( (\fbar_1,\mathcal{P}_1)(\fbar_2,\mathcal{P}_2) ) = (\prod_{z\in \mathcal{P}_1\ominus f_1(\mathcal{P}_2)} \psi_z) \phi(\fbar_1\fbar_2)\end{equation}
and on the other hand
 \begin{eqnarray}
 \chi (\fbar_1,\mathcal{P}_1)\chi(\fbar_1,\mathcal{P}_1)&=& (\prod_{z\in \mathcal{P}_1} \psi_z) \phi(\fbar_1)(\prod_{z\in \mathcal{P}_2} \psi_z) \phi(\fbar_2) \nonumber \\
&=&(\prod_{z\in \mathcal{P}_1} \psi_z) (\prod_{z\in f_1(\mathcal{P}_2)} \psi_z) \phi(\fbar_1)\phi(\fbar_2),\label{eq45}\end{eqnarray}
where the last identity follows from the equation $\phi(\fbar)\psi_z=\psi_{f(z)}\phi(\fbar)$.
The equality of the expressions in equations (\ref{eq44}) and (\ref{eq45}) now follows because $\phi $ is a homomorphism and because $\psi_z^2=1$.
This shows that $\chi $ is a homomorphism. To show that $\chi $ is injective, suppose that $\chi(\fbar,\mathcal{P}) $ is the identity automorphism. Then we have $\mathcal{P}=\emptyset$ and $\chi(\fbar,\emptyset)=\phi(\fbar)$, and from the  injectivity of $\phi$ we get that $\fbar=1$. Thus $\chi $ is injective.
  \end{proof}

{The theorem does not describe the whole group $\aut$ but only a subgroup, and it is not true in general that $\aut= \mgm$. However, the only cases we know where this equality does not hold are the surfaces corresponding to the acyclic types $\mathbb{D}_4$ and $\tilde{\mathbb{D}}_4$. These two types correspond to star shaped quivers with $3$ and $4$ branches which have $S_3$-symmetry and $S_4$-symmetry, respectively.  However the corresponding surfaces do not have such symmetries.
  We conjecture that these are the only exceptions, since we know no other surface that gives rise to a quiver having an $S_\ell$-symmetry, with $\ell>2$.
}

\begin{conj}\label{conj}
Let $(S,M)$ be any surface different from the disc with exactly one puncture and four marked points on the boundary or the disc with exactly two punctures and two marked points on the boundary. 
Then 
\begin{enumerate}
\item if $(S,M)$ is not a closed surface with exactly one puncture then
\[\aut= \mgm.\]
\item if $(S,M)$ is  a closed surface with exactly one puncture then
\[\aut= \mg.\]
\end{enumerate}
\end{conj}

We can prove this conjecture using the results from section \ref{sect 3.3} in the cases where the   cluster algebra from the surface is of acyclic type.

\begin{theorem}\label{theorem a}
Let $(S,M)$ be a disc or an annulus without punctures. Then \[\aut=\mg=\mgm.\]
\end{theorem}
\begin{proof} Since $(S,M)$ has no punctures, we have $\mg=\mgm$.
Suppose first that $(S,M)$ is a disc, and let the number of marked points be $n+3$. Then the cluster algebra is of Dynkin type $\mathbb{A}_n$, and we know from section \ref{sect 3.3} that $\aut\cong \ZZ_{n+3}$. Thus we only need to show that $\mg\cong\ZZ_{n+3}$.

We may assume without loss of generality that the marked points are the points of a  regular polygon, so that  any rotation about the centre of the disc of angle $k \frac{2\pi}{n+3}$, with $k\in \ZZ$, maps $M$ to $M$. Each of these rotations is an orientation preserving homeomorphism of $(S,M)$, and it is isotopic to the identity relative to $M$ if and only if it fixes each point in $M$. This shows that the rotations form a subgroup of $\mg$ isomorphic to $\ZZ_{n+3}$.

Each element of $\mg$ is determined by its values on $M$, because if two orientation preserving homeomorphisms $f,g$ agree on $M$, then $fg^{-1}$ fixes $M$, and we may suppose without loss of generality that $fg^{-1}\in \homeos$. Since $\mathcal{M}od(S,M)=0$, it follows that $fg^{-1}\in \textup{Homeo}_0(S,\partial S)$, and therefore $fg^{-1}\in \homeozero$. 
Since each element  of $\mg$ maps the boundary to itself and preserves orientation, each element can be represented by a rotation. This shows that $\mg\cong\ZZ_{n+3}$.

Now suppose that $(S,M)$ is an annulus, let $C_1,C_2$ be the two boundary components of $(S,M)$, and let $p$ be the number of marked points on $C_1$ and $ q$ be the number of marked points on $C_2$. Then the cluster algebra is of euclidean type $\tilde{\mathbb{A}}_{p,q}$, and we know from section \ref{sect 3.3} that 
\[\aut\cong \left\{ \begin{array}
{ll} H_{p,q} &\textup{if $p\ne q$;}\\
H_{p,p}\times \ZZ_2 &\textup{if $p= q$;}
\end{array}
\right.\]
where $H_{p,q}=\langle r_1,r_2\mid r_1r_2=r_2r_1, r_1^p=r_2^q\rangle.$
As in the case of the disc above, the rotations of each boundary component form a subgroup of $\mg$. Note however that these subgroups are infinite cyclic. We can choose two generators $r_1 $  for the group   given by rotating $C_1$ and   $r_2 $   for the group   given by rotating $C_2$ such that $r_1^p$ and $ r_2^q$ fix every point in $M$ and $r_1^p=r_2^q$. (Thus $r_1,r_2$ are rotations in opposite directions.) Moreover $r_1r_2=r_2r_1$. This shows that $H_{p,q}$ is a subgroup of $\mg$. 

Note that $r_1^p$ and $r_2^q$ are Dehn twists   of the annulus described in section \ref{sect 4.2}.

Suppose first that $p\ne q$. Then  each element of $\mg$ maps each boundary component to itself and, in particular, on each boundary component it is given by a rotation.
Moreover, each element of $\mg$  is determined by its values on $M$ up to composition with $r_1^p$, because if two elements $f,g$ agree on $M$, then $fg^{-1}$ fixes $M$, hence without loss of generality $fg^{-1}$ fixes each point on the boundary. It follows that $fg^{-1}\in\Mod$ and therefore $fg^{-1}$ is a power of a Dehn twist by Remark \ref{mgg}, hence a power of $r_1^p$. This shows that $\mg\cong H_{p,q}$.

Now suppose that $p=q$. Then the elements of $\mg$ may map one boundary component to the other. 
Exchanging the boundary components twice maps each boundary component to itself, and by the same argument as in the case $p\ne q$, we see that such an element is given by the rotations. Thus exchanging the boundary components corresponds to a subgroup of order two, whence $\mg\cong H_{p,p}\rtimes \ZZ_2$, as required.
 
\end{proof}

\begin{theorem}\label{theorem d}
Let $(S,M)$ be a disc with $p$ punctures with $p$ equal to $1$ or $2$, and suppose that the number of marked points on the boundary is at least $5$ if $p=1$ and at least $3$ if $p=2$. Then\[\aut= \mgm.\]

\end{theorem}
\begin{proof}
Suppose first that $(S,M)$ is a disc with one puncture  let $n$ be the number of marked points on the boundary. By our assumption, we have $n>4$. Then the cluster algebra is of type $\mathbb{D}_n$ and we know from section \ref{sect 3.3} that $\aut \cong \ZZ_n\times \ZZ_2$.
On the other hand, the mapping class group of the once punctured disc is equal to the mapping class group of the  unpunctured disc, see Remark \ref{rem mg2}, thus $\mg\cong\ZZ_n$.
Now it follows from Remark \ref{rem 38} that $\mgm\cong\ZZ_n\times \ZZ_2$ as required.

Suppose now that $(S,M)$ is a disc with two punctures, and let $n-3$ be the number of marked points on the boundary. By our assumption, we have $n>5$. Then the cluster algebra is of type $\tilde{\mathbb{D}}_{n-1}$, and we know from section \ref{sect_acyclic} that $\aut\cong G$, where
\begin{equation}\label{GG}
G=\left\langle \tau,\zs,\rho_1,\rho_n \left| \begin{array}{c}
 \rho_i^2=1 ,\tau \rho_i=\rho_i \tau \ (i=1,n)\\
\tau\zs=\zs\tau,\ \zs^2=\tau^{n-3} \\
\rho_1\zs=\zs\rho_n, \ \zs\rho_1=\rho_n\zs
\end{array}\right.\right\rangle
\end{equation}

The mapping class group $\mathcal{M}od$     of the disc with $p$ punctures (without any marked points on the boundary) is isomorphic to the braid group $B_p$ on $p$ strands, see Remark \ref{rem mg2}, thus in our situation it is isomorphic to $B_2\cong \ZZ$. Let $s$ be a generator of $B_2$. Then $s $ maps one puncture to the other and $s^2$ is isotopic to a rotation of the boundary by the angle $2\pi$ that fixes the punctures.

On the other hand, the elements of $\mg$ which are induced by the rotations of the boundary (fixing the punctures) form an infinite cyclic group, and we can choose a generator $r$ such that $r^{n-3}=s^2$. Clearly, $rs=s r$.

Up to composition with $r^{n-3}$ and $s$, any element of $\mg$ is determined by its values on $M$, because if $f,g\in \mg$ agree on $M$, then $fg^{-1}$ fixes each point in $M$, hence we can suppose without loss of generality that $fg^{-1}$ fixes each point on the
boundary. Thus since $\Mod$ is generated by $s$, it follows that $fg^{-1}$ is a power of $s$. 
This shows that $\mg=\langle r,s\mid s^2=r^{n-3}, rs=sr\rangle$.

We must show that $G\cong\mgm$. Denote the two punctures by $z_1,z_2$, and let  $\phi$ be the map from $G$ to $\mgm$ defined on the generators by 
\[\begin{array} {rcl}
\phi(\tau) &=&(r\,,\,\{z_1,z_2\})\\
\phi(\zs)&=& \left\{\begin{array}{ll}
(s,\emptyset ) &\textup{if $n$ odd}\\
(s,z_1) &\textup{if $n$ even}
\end{array}\right.\\
\phi(\rho_1) &=&(1,z_1)\\
\phi(\rho_n)&=& (1,z_2)
\end{array}
\]
and extended to $G$ by the homomorphism property.
One can easily check that $\phi $ preserves the relations of the group, for example if $n$ is even then
$$\phi(\zs^2)=(s,z_1)^2=(s^2,z_1\ominus s(z_1))=(s^2, \{z_1,z_2\}) $$
which is equal to
$$\phi(\tau^{n-3})=(r,\{z_1,z_2\})^{n-3} =(r^{n-3},\{z_1,z_2\}),$$ 
where the last identity follows from Corollary \ref{cor 37} (4), since $n$ is even.

To show that $\phi $ is injective suppose that $x=\tau^a\zs^b\rho_1^c\rho_n^d\in \textup{Ker}\, \phi$ for some integers $a,b,c,d$. Then 
$(1,\emptyset)=\phi(x)$, and by computing the first coordinate of this equation, we get $1=r^a s^b $. Consequently, since $\tau$ and $\sigma$ satisfy the same relations as $r$ and $s$, we have $1=\tau^a \zs^b$, and therefore $x=\rho_1^c\rho_n^d$.
Thus 
$$(1,\emptyset)=\phi(x)=(1,\{z_1\}^c\ominus\{z_2\}^d),$$
which implies that $c$ and $d$ are even, by Corollary \ref{cor 37}.(4), and thus $x=1$. This shows that $\phi$ is injective.

It remains to show that $\phi$ is surjective. Let $x=(r^as^b,\mathcal{P})\in \mgm$. Then $\phi(\tau^a\zs^b)=(r^as^b,\overline{\mathcal{P}})$, for some subset $\overline{\mathcal{P}}\subset \{z_1,z_2\}$, and multiplying with $\phi(\rho_1)$ or $\phi(\rho_2)$ if necessary, we see that $x$ lies in the image of $\phi$. This shows that $\phi $ is surjective, and
thus $\phi $ is an isomorphism.
 
\end{proof}

\begin{example}
We end this section with another look at Example \ref{ex torus}. The quiver $$\xymatrix@R15pt@C10pt{x_1\ar@<2pt>[rd]\ar@<-2pt>[rd] &&x_2 \ar@<2pt>[ll]\ar@<-2pt>[ll]\\
&x_3\ar@<2pt>[ru]\ar@<-2pt>[ru]
}$$ corresponds to a triangulation of the torus with one puncture, which can be seen easily using the plane as a universal cover and the triangulation shown on the left hand side of Figure \ref{fig torus}.
The edges are labeled $1,2,3$ instead of $x_1,x_2,x_3$ for brevity. Edges that have the same label are to be identified, and each point in Figure \ref{fig torus} is identified to the puncture. Thus in  the triangulation shown on the left hand side of Figure \ref{fig torus} there are exactly two triangles, both formed by edges $1,2,3$ and both having the same orientation. 

The picture in the middle of Figure \ref{fig torus} shows the triangulation corresponding to the seed obtained from the initial seed by mutating in $x_1 $, while
the image on the right hand side of Figure \ref{fig torus} shows the triangulation corresponding to the seed obtained   by mutating once more, this time in $x_2$. 
  
Geometrically, one can deform the picture on the left into the picture on the right by dragging the right end  upwards and the left end downwards. In the torus, this 
``deformation" corresponds to two Dehn twists along the closed curve labeled $3$. 
On the other hand, there is no orientation preserving homeomorphism transforming  the picture on the left into the one in the middle. Thus this mutation is not given by a mapping class. Of course we could have deduced this simply from the observation in Example \ref{ex torus} that this mutation   corresponds to an inverse cluster automorphism and not to a direct one.

\begin{figure}
\includegraphics{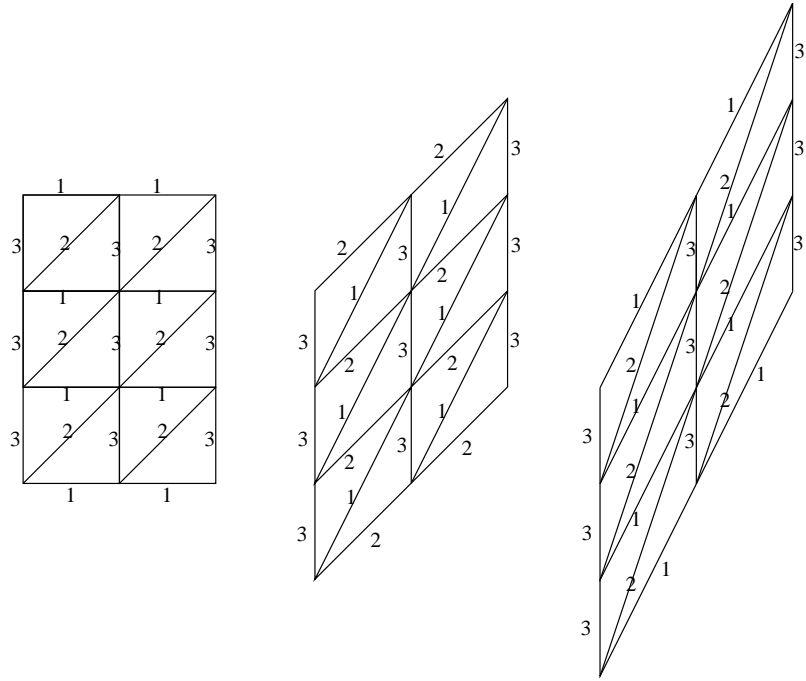}
\caption{Three triangulations of the torus}\label{fig torus}
\end{figure}
\end{example}

\end{section}
\begin{section}{Finiteness of the automorphism group}\label{sect 4}
In this section, we introduce the notion of \emph{automorphism finite} cluster algebras and  prove that for acyclic cluster algebras and for cluster algebras from surfaces it is equivalent to the notion of finite type cluster algebras. 

We say that a cluster algebra $\mathcal{A}$ is \emph{automorphism finite} if  its automorphism group $\Aut$ is finite.

\begin{theorem}\label{finite}
Let $\mathcal{A}$ be a cluster algebra arising from an acyclic quiver or from a surface. Then $\mathcal{A}$ is automorphism finite if and only if $\mathcal{A}$ is of Dynkin type.
\end{theorem}
\begin{proof}
 Sufficiency follows from Table \ref{table}. To prove necessity, suppose first that $\mathcal{A}$ is arising from an acyclic quiver $Q$. By Theorem \ref{theorem_main}, $\aut$ is isomorphic to the quotient of the group of automorphisms of the transjective component $\zG_{tr}$ of the Auslander-Reiten quiver of the cluster category $\mathcal{C}_Q$ modulo the stabiliser of the points in $\zG_{tr}$.
 
 If $Q$ is not of Dynkin type, then the Auslander-Reiten translation induces an element of $\textup{Aut}(\zG_{tr})$ of infinite order, acting freely on the points of $\zG_{tr}$. Thus $\Aut$ is infinite if $Q$ is acyclic and not Dynkin.
 
 Suppose now that $\mathcal{A}$ arises from a surface $(S,M)$. 
 By Lemma \ref{mgg}, the mapping class group $\Mod$ of the surface $S$ is a subgroup of $\mg$, which in turn is  isomorphic to a subgroup of $\Aut$, by Theorem \ref{theorem mg}. So in order to show that $\Aut$ is infinite, it suffices to find an infinite subgroup in $\Mod$.
  
 By Remark \ref{rem mg1}, there exists a Dehn twist which generates an infinite cyclic subgroup of $\Mod$ if $S$ has genus at least one or $S$ has genus zero and two or more boundary components.
  
There remain the cases where $S$ is a disc or a sphere. If $S$ is a disc with $p\ge 2$ punctures, then the braid group $B_p$ is an infinite subgroup of the mapping class group of $S$.
 In the cases where $p=0$ or $1$, we have that $\mathcal{A}$ is of Dynkin type $\mathbb{A}$ or $\mathbb{D}$, respectively. 
 
 Finally, if $S$ is a sphere with $p$ punctures, then by our assumption $p\ge 4$, and then it is known that the mapping class group of $S$ has a free subgroup, see \cite[4.2]{FM}.
\end{proof}

\begin{remark}
For a sphere with 3 punctures, the mapping class group is $S_3$, which is a finite group. However the sphere with 3 or less punctures is excluded in the construction of cluster algebras from surfaces in \cite{FST}.
\end{remark}
\end{section}

\affiliationone{
   Ibrahim Assem and Vasilisa Shramchenko\\
   D\'epartement de math\'ematiques\\
   Universit\'e de Sherbrooke\\
2500, boul. de l'Universit\'e,\\
Sherbrooke, Qu\'ebec,
J1K 2R1\\
   Canada
   \email{Ibrahim.Assem@USherbrooke.ca\\
   Vasilisa.Shramchenko@USherbrooke.ca}}
   \affiliationtwo{
   Ralf Schiffler\\
   Department of Mathematics \\
196 Auditorium Road\\
University of Connecticut, U-3009\\
Storrs, CT 06269-3009\\
USA
\email{schiffler@math.uconn.edu}}
\end{document}

%% file: figtag.pstex_t
\begin{picture}(0,0)%
\includegraphics{figtag.pstex}%
\end{picture}%
\setlength{\unitlength}{3947sp}%
\begingroup\makeatletter\ifx\SetFigFont\undefined%
\gdef\SetFigFont#1#2#3#4#5{%
  \reset@font\fontsize{#1}{#2pt}%
  \fontfamily{#3}\fontseries{#4}\fontshape{#5}%
  \selectfont}%
\fi\endgroup%
\begin{picture}(6212,3125)(787,-2530)
\put(5753,-2080){\makebox(0,0)[lb]{\smash{{\SetFigFont{12}{14.4}{\rmdefault}{\mddefault}{\updefault}{\color[rgb]{0,0,0}$\alpha^z$}%
}}}}
\put(904,-223){\makebox(0,0)[lb]{\smash{{\SetFigFont{12}{14.4}{\rmdefault}{\mddefault}{\updefault}{\color[rgb]{0,0,0}$z$}%
}}}}
\put(904,-2392){\makebox(0,0)[lb]{\smash{{\SetFigFont{12}{14.4}{\rmdefault}{\mddefault}{\updefault}{\color[rgb]{0,0,0}$z$}%
}}}}
\put(4664,-1161){\makebox(0,0)[lb]{\smash{{\SetFigFont{12}{14.4}{\rmdefault}{\mddefault}{\updefault}{\color[rgb]{0,0,0}$z$}%
}}}}
\put(4664,-223){\makebox(0,0)[lb]{\smash{{\SetFigFont{12}{14.4}{\rmdefault}{\mddefault}{\updefault}{\color[rgb]{0,0,0}$z$}%
}}}}
\put(4664,-2392){\makebox(0,0)[lb]{\smash{{\SetFigFont{12}{14.4}{\rmdefault}{\mddefault}{\updefault}{\color[rgb]{0,0,0}$z$}%
}}}}
\put(1933,412){\makebox(0,0)[lb]{\smash{{\SetFigFont{12}{14.4}{\rmdefault}{\mddefault}{\updefault}{\color[rgb]{0,0,0}$\alpha$}%
}}}}
\put(1993,-550){\makebox(0,0)[lb]{\smash{{\SetFigFont{12}{14.4}{\rmdefault}{\mddefault}{\updefault}{\color[rgb]{0,0,0}$\alpha$}%
}}}}
\put(1993,-2080){\makebox(0,0)[lb]{\smash{{\SetFigFont{12}{14.4}{\rmdefault}{\mddefault}{\updefault}{\color[rgb]{0,0,0}$\alpha$}%
}}}}
\put(5693,412){\makebox(0,0)[lb]{\smash{{\SetFigFont{12}{14.4}{\rmdefault}{\mddefault}{\updefault}{\color[rgb]{0,0,0}$\alpha^z$}%
}}}}
\put(5753,-550){\makebox(0,0)[lb]{\smash{{\SetFigFont{12}{14.4}{\rmdefault}{\mddefault}{\updefault}{\color[rgb]{0,0,0}$\alpha^z$}%
}}}}
\put(904,-1161){\makebox(0,0)[lb]{\smash{{\SetFigFont{12}{14.4}{\rmdefault}{\mddefault}{\updefault}{\color[rgb]{0,0,0}$z$}%
}}}}
\end{picture}%